\documentclass[11pt]{amsart}

\usepackage[marginratio=1:1,height=600pt,width=404pt,tmargin=110pt]{geometry}

\usepackage{amsmath, amsfonts, amssymb, amscd, bezier, amstext, amsthm}
\usepackage{color}
\usepackage{graphicx}
\usepackage{fancyhdr}
\usepackage{hyperref}

\DeclareMathOperator{\Div}{div}

\DeclareMathOperator{\Ric}{Ric}
\DeclareMathOperator{\Span}{span}
\DeclareMathOperator{\vol}{vol}
\DeclareMathOperator{\supp}{supp}

\DeclareMathOperator{\ind}{ind}
\DeclareMathOperator{\Area}{Area}
\DeclareMathOperator{\id}{id}
\DeclareMathOperator{\Null}{null}
\DeclareMathOperator{\dvol}{dvol}

\newcommand{\R}{\mathbb{R}}
\newcommand{\N}{\mathbb{N}}

\newcommand{\C}{\mathbb{C}}
\newcommand{\e}{\varepsilon}

\theoremstyle{plain}

\newtheorem{theorem}{Theorem}[section]
\newtheorem{lemma}[theorem]{Lemma}
\newtheorem{corollary}[theorem]{Corollary}
\newtheorem{proposition}[theorem]{Proposition}

\newtheorem{remark}[theorem]{Remark}
\newtheorem{definition}[theorem]{Definition}

\title[Solutions of the Allen-Cahn equation in the presence of symmetry]{Solutions of the Allen-Cahn equation on closed manifolds in the presence of symmetry}

\author[R. Caju]{Rayssa Caju}

\author[P. Gaspar]{Pedro Gaspar}

\address{Department of Mathematics, The University of Chicago,	5734 S University Ave, Chicago IL 60637, USA}
\email{\href{mailto:rayssacaju@gmail.com}{rayssacaju@gmail.com},\href{mailto:pgaspar@uchicago.edu}{pgaspar@uchicago.edu}}

\date{}

\numberwithin{equation}{section}

\begin{document}
	
%%%%%%%%%%%%%%%%%%%%%%%%%%%%%%%%%%%%%%%%%%%%%%%%%%%%%%%%%%%%%%%%%%%%%%%%%%%%%%%%%%%%%%%%%%%%%%%%%%%%%%%%%%%%%%%%%%%%%%%%%%%%%%%%%%%%%%%%%%%%%%%%%%%%%%%%%%%%%%%%
%
%                               ABSTRACT
%
%%%%%%%%%%%%%%%%%%%%%%%%%%%%%%%%%%%%%%%%%%%%%%%%%%%%%%%%%%%%%%%%%%%%%%%%%%%%%%%%%%%%%%%%%%%%%%%%%%%%%%%%%%%%%%%%%%%%%%%%%%%%%%%%%%%%%%%%%%%%%%%%%%%%%%%%%%%%%%%%

\begin{abstract} We prove that given a minimal hypersurface $\Gamma$ in a compact Riemannian manifold $M$ without boundary, if all the Jacobi fields of $\Gamma$ are generated by ambient isometries, then we can find solutions of the Allen-Cahn equation $-\e^2\Delta u +W'(u)=0$ on $M$, for sufficiently small $\e>0$, whose nodal sets converge to $\Gamma$. This extends the results of Pacard-Ritor\'e \cite{PR} (in the case of closed manifolds and zero mean curvature). 

\end{abstract}

\maketitle

%%%%%%%%%%%%%%%%%%%%%%%%%%%%%%%%%%%%%%%%%%%%%%%%%%%%%%%%%%%%%%%%%%%%%%%%%%%%%%%%%%%%%%%%%%%%%%%%%%%%%%%%%%%%%%%%%%%%%%%%%%%%%%%%%%%%%%%%%%%%%%%%%%%%%%%%%%%%%%%%
%
%                                                   CONTENTS
%
%%%%%%%%%%%%%%%%%%%%%%%%%%%%%%%%%%%%%%%%%%%%%%%%%%%%%%%%%%%%%%%%%%%%%%%%%%%%%%%%%%%%%%%%%%%%%%%%%%%%%%%%%%%%%%%%%%%%%%%%%%%%%%%%%%%%%%%%%%%%%%%%%%%%%%%%%%%%%%%%

%\tableofcontents

%%%%%%%%%%%%%%%%%%%%%%%%%%%%%%%%%%%%%%%%%%%%%%%%%%%%%%%%%%%%%%%%%%%%%%%%%%%%%%%%%%%%%%%%%%%%%%%%%%%%%%%%%%%%%%%%%%%%%%%%%%%%%%%%%%%%%%%%%%%%%%%%%%%%%%%%%%%%%%%%
%
%                                               INTRODUCTION
%
%%%%%%%%%%%%%%%%%%%%%%%%%%%%%%%%%%%%%%%%%%%%%%%%%%%%%%%%%%%%%%%%%%%%%%%%%%%%%%%%%%%%%%%%%%%%%%%%%%%%%%%%%%%%%%%%%%%%%%%%%%%%%%%%%%%%%%%%%%%%%%%%%%%%%%%%%%%%%%%%

\section{Introduction} \label{intro} 	

Let $(M,g)$ be a $(n+1)$-dimensional compact Riemannian manifold without boundary. We are interested in studying solutions of the \textit{Allen-Cahn equation} 
\begin{equation} \label{AC}
	\varepsilon^{2}\Delta u - W'(u) = 0 \quad \mbox{ in } \quad M
\end{equation} 
whose nodal sets accumulate around a minimal hypersurface $\Gamma \subset M$ which is non-degenerate up to ambient isometries. Here $\varepsilon > 0$ is a small parameter and $W$ is a double-well potential, such as $W(u)={(1-u^2)^2}/{4}$. 

The Allen-Cahn equation models phase transition phenomena \cite{AllenCahn}, and it is closely related to the theory of minimal surfaces. This analogy, which dates back to De Giorgi \cite{DeGiorgi}, has been explored in many contexts since the works of Gurtin \cite{Gurtin}, Modica \cite{Modica} and Sternberg \cite{Sternberg}, among others, and it has proved to be quite fruitful in the study of elliptic equations of this nature, as well as in the investigation of certain properties of minimal surfaces. We refer to the surveys \cite{ChanWei} and \cite{PA}, and the references therein.

This analogy may be described from a variational viewpoint as follows. Solutions of \eqref{AC} are critical points of the \emph{energy functional} given by
\begin{equation*}
	{E}_{\varepsilon}(u) := \int_{M} \left(\frac{\varepsilon}{2}|\nabla u|^{2} + \frac{W(u)}{\e}\right) \dvol_{g}.
\end{equation*}
Local or constrained minimizers of this functional approach the two minima of $W$ in large regions of $M$, representing two different phase in the physical model, while keeping the transition region as small as possible, for it is penalized by the term $\e|\nabla u|^2$. As the weight $\e$ decreases, we expect to get a \emph{limit interface} $\Gamma \subset M$ which minimizes the area functional in a suitable sense.

In other words, under certain conditions, we expect that $\Gamma$ is a critical point of the $n$-dimensional volume, or the \emph{Area functional}, and that the energy of such solutions converges to $\Area(\Gamma)=\mathcal{H}^n(\Gamma)$. This means that $E_\e$ may be regarded as a singular perturbation of $\Area$. See \cite{HutchinsonTonegawa} for a more precise description of this convergence in the case of solutions with bounded energy, using tools of Geometric Measure Theory.

%Let $\Gamma \subset M$ be an oriented minimal hypersurface which separates $M$ in the following sense. The hypersurface $\Gamma$ is the zero set of a smooth function $f_{\Gamma}$ defined on $M$ and for which $0$ is a regular value. Then $M\backslash \Gamma = M_{+}\cup M_{-}$, where 
%\begin{equation*}
%	M_{-} = f_{\Gamma}^{-1}((-\infty, 0)) \quad  \mbox{ and } \quad M_{+} = f_{\Gamma}^{-1}((0,+\infty)).
%\end{equation*}
%In particular $\Gamma$ must be oriented, and we may consider the an unit normal vector field $\nu$ on $\Gamma$ pointing towards $M_{+}$. 

Conversely, given a critical point $\Gamma \subset M$ of the area functional, one may ask whether $\Gamma$ can be obtained as a limit interface of solutions $u_\e$ of \eqref{AC}, for small $\e>0$. The pioneering work of F. Pacard and M. Ritor\'e \cite{PR} answers positively this question in a variety of situations. Let us describe their result in our setting, where $\Gamma$ is a minimal hypersurface and $M$ has empty boundary.

Recall that the \emph{Jacobi operator} of the minimal hypersurface $\Gamma$ is given by
\begin{equation*}
	J_{\Gamma} = \Delta_{\Gamma} + |A_\Gamma|^{2} + \Ric(\nu,\nu)
\end{equation*}
where $\nu$ is a unit normal vector field along $\Gamma$, $\Ric$ denotes the Ricci tensor on $(M,g)$, $\Delta_{\Gamma}$ is the Laplace-Beltrami operator on $\Gamma$, and $|A_\Gamma|$ is the norm of the second fundamental form of $\Gamma$. This elliptic operator arises naturally in the analysis of the stability of $\Gamma$: if $X=\phi\nu$ is a normal vector field on $\Gamma$, where $\phi \in C^ \infty(\Gamma)$, then the \emph{second variation of the area} of $\Gamma$ in the direction of $X$ is the quantity
	\[\delta^2\Gamma(X) = \frac{d}{dt}\bigg|_{t=0}\Area(\psi^t(\Gamma)) = \int_\Gamma |\nabla_\Gamma \phi|^ 2 - (|A_\Gamma|^2 + \Ric(\nu,\nu))\phi^2\,\mathrm{d}\Gamma\]
where $\psi^t$ is the flow generated by a local extension of $X$ to a neighborhood of $\Gamma$ in $M$. Hence $J_\Gamma$ is the elliptic operator associated to the quadratic form $Q(\phi)=-\delta^2\Gamma(\phi \nu)$.

A smooth function $\phi$ on $M$ which solves $J_\Gamma\phi =0$ is called a \emph{Jacobi field}. We say $\Gamma$ is a \emph{nondegenerate} minimal hypersurface if $J_\Gamma$ has trivial kernel. For a generic Riemannian metric on $M$ (in the Baire sense), any closed minimal hypersurface in $M$ has this property, as proved by B. White in \cite{White1, White2}. 

Let $\Gamma \subset M$ be a nondegenerate minimal hypersurface which separates $M$ in the following sense: there is a smooth function $f_\Gamma$ on $M$ such that $0$ is a regular value for $f_\Gamma$ and $\Gamma=f_\Gamma^{-1}(0)$. Pacard and Ritor\'e (see also \cite{PA}) employed an infinite dimensional Lyapunov-Schmidt reduction technique to show that one can find a family $\{u_\e\}$ of solutions of \eqref{AC}, for sufficiently small $\e>0$, such that the nodal sets $u_\e^{-1}(0)$ converge to $\Gamma$ and $2\sigma E_\e(u_\e) \to \Area(\Gamma)$, where $\sigma>0$ is a constant depending only on $W$.

There are related works which tackle cases where $J_\Gamma$ fails to be injective. In the noncompact case, M. del Pino, M. Kowalczyk and J. Wei \cite{PKW} produced solutions of \eqref{AC} in $\R^3$ with level sets which accumulate around embedded minimal surfaces with finite total curvature in $\R^3$ assuming that all bounded Jacobi fields are originated from a set of rigid motions. Parallel to this issue is the construction due to del Pino-Kowalczyk-Wei \cite{PKW-DG} of solutions in $\R^N$, for $N\geq 9$, whose level sets accumulate around a singular minimal cone, and which gives a counterexample, in these dimensions, to the well-know conjecture of De Giorgi \cite{DeGiorgi,ChanWei} concerning entire solutions of \eqref{AC}.
 
More generally, whenever $(M,g)$ has a one-parameter family of isometries which does not preserve $\Gamma$, the operator $J_\Gamma$ has nontrivial kernel. In fact, any such family generates a \emph{Killing vector field} $X$ on $M$, whose flow preserves the area of $\Gamma$ and consequently restricts to a normal vector field $\left\langle X,\nu\right\rangle \nu$ with $J_\Gamma(\left\langle X,\nu\right\rangle)=0$. In this case, we say that $\langle X, \nu \rangle$ is a \emph{Killing-Jacobi field}.

Based on the results of \cite{PKW-DG,PKW}, it should be possible to construct solutions of \eqref{AC} whose level sets accumulate around a given separating minimal hypersurface $\Gamma$ in $M$ assuming only that all Jacobi fields are generated by global isometries in the sense described above. More precisely, it should suffice to assume that there exist Killing fields $X_{1},\dots, X_{d}$ on $M$ such that 
\begin{equation} \label{killjacobi}
	\ker(J_{\Gamma}) = \Span \{z_{i}\}_{i=1}^{d}, \qquad \mbox{where} \qquad z_{i} = \left\langle X_{i}, \nu\right\rangle.
\end{equation}

Our main result confirms this expectation. We prove

\begin{theorem}\label{main} Let $(M,g)$ be a $(n+1)$-dimensional compact Riemannian manifold without boundary and let $\Gamma \subset M$ be a minimal hypersurface which separates $M$, with $M\backslash \Gamma = M_{+}\cup M_{-}$. Suppose that all Jacobi fields are generated by global isometries of $M$. Then there exists $\varepsilon_{0}>0$ such that for all $\varepsilon \in (0,\varepsilon_{0})$ there is a solution $u_{\varepsilon}$ of \eqref{AC} such that $u_\e$ converges uniformly to 1 (respectively to $-1$) on compact subsets of $M_{+}$ (respectively $M_{-}$), and
	\begin{equation*}
		{E}_{\varepsilon}(u_{\varepsilon}) \rightarrow \frac{1}{2\sigma}\Area(\Gamma), \qquad \mbox{as} \quad \e \to 0,
	\end{equation*}
where $\sigma = \int_{-1}^1\sqrt{W(t)/2}\,dt$. Moreover, the Morse index $m(u_\e)$ and the nullity $n(u_\e)$ of $u_\e$ satisfy
	\[m(u_\e) = \ind(J_\Gamma) \qquad \mbox{and} \qquad n(u_\e) = \Null(J_\Gamma), \] 
where $\ind(J_\Gamma)$ is the Morse index of $\Gamma$ (the number of negative eigenvalues of $J_\Gamma$) and $\Null(J_\Gamma) = \dim \ker(J_\Gamma)$.
\end{theorem}

In some cases, one may adapt the techniques of Pacard and Ritor\'e to construct solutions under these conditions by working on subspaces of functions which are invariant by a finite group of symmetries which \emph{do not} leave the Jacobi fields invariant. As pointed out in \cite{PR}, this approach can be used when $(M^{n+1},g)$ is the $(n+1)$-dimensional round sphere and $\Gamma$ is an equator, or when $(M^{n+1},g)$ is a flat torus and $\Gamma$ is pair of parallel meridians. In these examples, one can easily describe all Jacobi fields and then use reflections on hyperplanes to reduce to a nondegenerate problem.

Nevertheless, even for arbitrary minimal hypersurfaces in $S^{n+1}$, the description of Jacobi fields might be considerably more complicated (see e.g. \cite{KapouleasWiygul} for the description of the Jacobi fields of some minimal surfaces in $S^3$ constructed by Lawson \cite{Lawson}), so one may look for a different strategy which does not require an explicit description of $\ker(J_\Gamma)$.  We describe other examples of families of minimal surfaces for which our result can be applied in Section \ref{applications}.%Note that all \emph{embedded} minimal surfaces in $S^3$ are expected to satisfy \eqref{killjacobi}, see \cite{odnadosnvs}.

Additionally, by the nature of these solutions, one may expect to devise a sufficiently robust strategy to allow for perturbations of the Riemannian metric away from $\Gamma$. We confirm this expectation in some particular cases; see Theorem \ref{resposta} for a precise statement. In particular, we prove the existence of solutions which accumulate around the equator in $(S^{n+1}, \tilde g)$, where $\tilde g$ is a perturbation of the round metric in a neighborhood of the north pole, giving a positive answer to a question posed by Pacard and Ritor\'e \cite[Section 5]{PR}.

\begin{remark}
In \cite{PR}, Pacard and Ritor\'e prove that if $\partial M \neq \emptyset$, then any separating \emph{free boundary} minimal hypersurface which is nondegenerate with respect to variations satisfying a boundary condition can be obtained as the limit interface of solutions $u$ of \eqref{AC} satisfying $\partial_{\nu_{\partial M}}u=0$ along $\partial M$. More generally, one can produce solutions of $\e^2\Delta u - W'(u) = \e \lambda$ which accumulate around any separating constant mean curvature hypersurface which is nondegenerate with respect to volume-preserving variations. It is reasonable to expect that the approach of the current paper can be used to relax these nondegeneracy hypotheses, in a similar spirit to Theorem \ref{main}.
\end{remark}

We mention here some related existence results. Agudelo-del Pino-Wei \cite{Agudelo2} established the existence of solutions in bounded domains in $\R^3$ whose level sets resemble a bounded portion of a finite total curvature minimal surface in $\R^3$. In \cite{PKWY}, del Pino-Kowalczyk-Wei-Yang constructed solutions which accumulate on a nondegenerate separating minimal hypersurface with multiple transition layers, assuming a positivity condition on $\Ric$. A similar result, for entire solutions in $\R^3$ with multiple catenoidal ends was obtained in \cite{Agudelo1}. Solutions with small Dirichlet boundary data were constructed by Chodosh and Mantoulidis in \cite{OC}. In \cite{Eduardo}, Hitomi obtained solutions of a nonlinear Neumann boundary problem associated to \eqref{AC} with level sets concentrating on a nondegenerate capillary hypersurface. 

A different strategy to address the existence of solutions relies on the variational perspective. In \cite{Guaraco}, M. Guaraco employed min-max methods to produce solutions of \eqref{AC} in closed Riemannian manifolds, obtaining a phase transitions based proof of the celebrated result of Almgren and Pitts \cite{Almgren,Pitts} on the existence of minimal hypersurfaces in such manifolds. This construction can be seen as an instance of the parallels between the Allen-Cahn equation and minimal surfaces within a min-max framework, and it was further extended by the second-named author and Guaraco in \cite{GasparGuaraco} (see also \cite{Passaseo} for an earlier related construction, in bounded Euclidean domains), inspired by the remarkable developments of \emph{Almgren-Pitts min-max theory} for minimal hypersurfaces. This theory provided answers to notable geometric problems such as the Willmore Conjecture \cite{Willmore} and Yau's Conjecture (see \cite{IMN,MNS,Song} and the references therein), and culminated in the development of a Morse theoretic description of the space of minimal hypersurfaces, as proposed by F. Marques and A. Neves \cite{MNMult}. 

As discussed in \cite{MNMult}, one of the main ingredients of this Morse-theoretic description is a proof the \emph{Multiplicity One Conjecture}: for generic metrics, unstable components of min-max minimal hypersurfaces should occur with multiplicity one, see \cite{MNIndex}. In the Allen-Cahn setting, a proof of a strong version of this conjecture in 3-dimensional manifolds was carried out by Chodosh and Mantoulidis in \cite{OC}, where they derive curvature and strong sheet separation estimates for stable solutions (bulding on the previous work of Wang and Wei \cite{WangWei}) and obtain another proof of Yau's conjecture on 3-manifolds with generic metrics. Finally, we point out that the Multiplicity One Conjecture was also recently solved, in the Almgren-Pitts setting, by X. Zhou \cite{Zhou}, and yet another proof of Yau's Conjecture using Allen-Cahn methods was obtained in \cite{GGWeyl}, based on \cite{IMN}, \cite{MNS}, and on the Weyl asymptotic Law for critical values for the area functional proved in \cite{LMN}.

%In a similar direction, but using different techniques, Gui-Liu-Wei \cite{GuiLiuWei} produced entire two-ended solutions of \eqref{AC} in $\R^3$

Let us briefly describe the proof of Theorem \ref{main}. We overcome the lack of injectivity of $J_\Gamma$ employing an approach based on \cite{PKW}. We use the infinite-dimensional Lyapunov-Schmidt reduction strategy from \cite{PA,PR} to produce solutions $u_\e$ of
	\begin{equation} \label{ACj}
		-\e^2\Delta u_\e + W'(u_\e) = \e f_\e
	\end{equation}
for sufficiently small $\e>0$, where $f_\e$ vanishes outside a neighborhood of $\Gamma$, and it is roughly given as the product of the derivative of one-dimensional approximations of the solution obtained using heteroclinic solutions of \eqref{AC} in $\R$, and a linear combination of the Jacobi fields of $\Gamma$ in a tubular neighborhood of this hypersurface. Using the functions $\langle X_1, \nabla u_\e\rangle, \ldots, \langle X_d, \nabla u_\e\rangle$ as test functions for \eqref{ACj} we prove that $f_\e$ vanishes, so $u_\e$ is a solution of \eqref{AC}. A related strategy was used in \cite{BPS} to produce deformations of compact free boundary constant mean curvature surfaces. The description of the Morse index and the nullity of $u_\e$ is a consequence of the index estimates of \cite{GasparInner}, and the aforementioned work of Chodosh and Mantoulidis \cite{OC}, where they characterize the Morse index of multiplicity one limit interfaces.% and their multiplicity.

%For any sequence of solutions $\{u_i\}$ of \eqref{AC} for $\e = \e_i \downarrow 0$ with bounded energy and bounded Morse index, the aforementioned results of Hutchinson and Tonegawa \cite{HutchinsonTonegawa}, and the works of Tonegawa-Wickramasekera \cite{Tonegawa,TonegawaWickramasekera} and Guaraco \cite{Guaraco} imply that $2\sigma E_\e(u_\e) \to \sum_{i=1}^N m_i\Area(\Gamma_i)$, where $\Gamma_1,\ldots, \Gamma_N$ are the components of the limit interface $\Gamma$, and $m_1,\ldots, m_N$ are positive integers, called the \emph{multiplicities} of these components. The recent work of O. Chodosh and C. Mantoulidis \cite{OC} show that, in the multiplicity one case-- in particular, for the solutions given by Theorem \ref{main} -- the stability properties of $\Gamma$ are connected to the corresponding properties of the solutions, in the sense that, for small $\e>0$, the Morse index of such solutions (viewed as critical points of the energy functional) coincides with the Morse index of $\Gamma$ (as critical points of the area functional).

The paper is organized as follows. In Section \ref{pre} we introduce preliminary notation and some results concerning Fermi coordinates and one dimensional solutions of \eqref{AC}. In Section \ref{reduction} we follow the strategy of \cite{PA} to decompose the Allen-Cahn equation on $M$ as a coupled system of non-linear equations. In Section \ref{proof} we conclude the proof Theorem \ref{main} using a fixed-point argument and variations by the Killing fields as test functions in \eqref{AC}. In Section \ref{applications} we state some examples of minimal surfaces satisfying the hypothesis of this theorem and state some consequences of our main result. %%%%%%%%%%%%%%%%%%%%%%%%%%% OUTLINE! %%%%%%%%%%%%%%%%%%%%%%%%%%%%%%%%%%%%R

\subsection{Acknowledgments}
The problem addressed in this paper was suggested by Fernando Cod\'a Marques. We are very grateful for his interest and his continued encouragement and support. We would also like to thank Andr\'e Neves and Marco Guaraco for many valuable discussions and for helpful comments on the first draft of this article.

\section{Preliminaries} \label{pre}

\subsection{Fermi coordinates} \label{fermi}

We recall the definition of the \emph{Fermi coordinates} introduced in \cite{PA}. By our hypothesis on $\Gamma$, this hypersurface is the boundary of an open region $M_+$ in $M$. Consequently, $\Gamma$ is orientable. Denote by $\nu$ the unit normal vector field on $\Gamma$ which points towards $M_+$. We use the exponential map of $M$ to parametrize a tubular neighborhood of $\Gamma$ as
\begin{equation*}
Z(y,z) := \exp_{y}(z\nu(y))
\end{equation*}
where $y \in \Gamma$ and $z\in\R$. This defines a diffeomorphism 
	\[Z:\Gamma \times (-\tau,\tau) \to \mathcal{V}\]
onto a neighborhood $\mathcal{V}$ of $\Gamma$ in $M$, for some $\tau>0$. We denote by $\Gamma_z$ the parallel surfaces to $\Gamma$ at height $z$, that is $\Gamma_z = \{Z(y,z) : y \in \Gamma\}$. In this coordinate system, we have the following expansions for the Riemannian metric $g$
	\[Z^*g = g_z + dz^2,\]
where denote by $g_z$ the pullback of the metric induced on $\Gamma_z$, considered as a family of metrics on $\Gamma$ parametrized by the height $z \in (-\tau,\tau)$. If we denote by $\partial_z$ the vector field ${\partial_z}|_{Z(y,z)} = DZ_{(y,z)}(0,1)$, then the gradient of a function $f$ on $M$ may be expressed as
	\[\nabla f = \nabla^{g_{z}}f + (\partial_{z}f) \cdot \partial_z,\]
where the first term is the gradient of $f|_{\Gamma_z}$ along $\Gamma_z$. Finally, the Laplacian of any $f$ of class $C^2$ may be written as
	\[\Delta f = \partial_z(\partial_z f) + \Delta_{g_z}f - H_z (\partial_z f),\]
where $H_z$ is the mean curvature of $\Gamma_z$ and we regard $\Delta_{g_z}f$ as above. See Section 3.1 in \cite{PA} for the asymptotic expansion of these quantities in terms $z$ and the geometry of $M$ and $\Gamma$.

\subsection{Hypotheses on \texorpdfstring{$W$}{W} and heteroclinic solutions}
Along the paper, we will assume that $W \in C^3(\R)$ is an \emph{even double-well potential}, meaning
	\begin{enumerate}
		\item $W\geq 0$ and it vanishes only at $\pm 1$, which are non-degenerate global minima, that is $W''(\pm 1)>0$.
		\item $0 \in \R$ is a local maximum with $W''(0)<0$, and $W'(t)<0$ for $t \in (0,1)$.
	\end{enumerate}
The standard example of such a function is $W(t) = \frac{(1-t^2)^2}{4}$.

The solutions constructed in Theorem \ref{main} are based on the \emph{heteroclinic} solution of the one-dimensional Allen-Cahn equation. Namely, there is $\psi_{1} \in C^3(\R)$ which solves
\begin{equation}
	\label{het}-\psi''+W'(\psi)=0, \qquad \mbox{and} \qquad \psi(0)=0,
\end{equation}
and satisfies $\psi_{1}(z)\to \pm 1$ as $z \to \pm \infty$. It solves the first order ODE $\psi_1'=\sqrt{2W(\psi_1)}$, and it has finite energy, for it holds
	\[\sigma = \int_\R (\psi_{1}'(z))^2\,dz<+\infty,\]
where $\sigma := \int_{-1}^1 \sqrt{W(t)/2}\,dt$. Given $\e>0$, we let $\psi_\e(z)=\psi_{1}(z/\e)$. Note that $\psi_\e$ solves \eqref{AC} on the real line. We denote
	\[
	 \dot{\psi}_\e(z) = \psi_{1}'(z/\e), \qquad \mbox{and} \qquad  \ddot{\psi}_\e(z) = \psi_{1}''(z/\e).
	\]
Observe that
	\[w_\e(x_1,\ldots, x_{n+1}) = \psi_\e(x_{n+1})\]
is a solution of \eqref{AC} in $\R^{n+1}$ which has the hyperplane $\{x_{n+1}=0\}$ as its zero set, and it approaches $\pm 1$ as $|x_{n+1}|$ increases. 

\begin{remark}\label{rem}
Let $\delta_*\in (0,1)$, $c>0$ a positive constant and consider $\chi$ be a smooth cut-off function with  $\chi(z) = 1$ for $z \leq c\varepsilon^{\delta_*}$. Then,
\begin{equation*}
\frac{1}{\varepsilon}\int_{\R}\dot{\psi}_{\varepsilon}^{2}(z)\chi(z)dz = \frac{1}{\varepsilon}\int_{\R}\dot{\psi}_{\varepsilon}^{2}(z)dz + o(1) = \sigma + o(1)
\end{equation*}
when $\varepsilon\downarrow 0$.
\end{remark}

\section{An infinite dimensional reduction} \label{reduction}

\subsection{Linear operators and weighted H\"{o}lder spaces} \label{linear}

As previously mentioned, we follow Pacard \cite{PA} and apply the Lyapunov-Schmidt reduction argument in order to reduce the problem to a fixed point problem for a coupled system of equations. These equations are described in terms of three linear operators: an operator $L_\e$ acting in functions defined in $\Gamma\times \R$, an operator $\mathcal{L}_\e$ acting on functions defined on the whole manifold $M$, and finally the Jacobi operator of $\Gamma$. Moreover, the appropriate estimates for $L_\e$ and $\mathcal{L}_\e$ require the use of \emph{weighted H\"{o}lder spaces}. We briefly recall the definition of such spaces, introduce $L_\e$ and $\mathcal{L}_\e$ and recall the existence theory for these operators.

\begin{definition} For all $k \in \N$ and $\alpha \in (0,1)$, denote by $C^{k,\alpha}_{\varepsilon}(\Gamma\times\R)$ the space of functions $w \in C^{k,\alpha}_{loc}(\Gamma\times\R)$, where the H\"{o}der norm is computed with respect to the scaled metric $\varepsilon^{2}(g_{0} + dt^{2})$, that is 
	\[\|w\|_{C^{k,\alpha}_{\varepsilon}} := \sum_{j=0}^{k}\varepsilon^{j}\|\nabla^{j}w\|_{L^{\infty}} + \varepsilon^{k+\alpha}\sup _{x \neq y} \frac{\left|\nabla^{k} w(x)-\nabla^{k} w(y)\right|}{{d}(x, y)^{\alpha}}.\]
In the expression above, $\nabla$ denotes the Riemannian connection and $d(x,y)$ is the Riemannian distance, both computed with respect to the product metric $g_{0} + dt^{2}$.
\end{definition}

Observe that if $w \in C^{k,\alpha}_{\varepsilon}(\Gamma\times\R)$ then 
	\begin{align*}
		\|\nabla_\Gamma^{a}\,\partial_{t}^{b}\, w\|_{L^\infty} &\leq \varepsilon^{-a-b}\|w\|_{C^{k,\alpha}_{\varepsilon}}
	\end{align*}
provided $a+b\leq k$.

By the spectral analysis of the linearization of the operator in \eqref{het} at $\psi_{1}$, which is given by $L_{0} = \partial_{t}^{2} - W''(\psi_1)$ and which annihilates $\psi_1'$, it follows that the restriction of the quadratic form associated with $L_{0}$ to the space of functions which are $L^{2}$-orthogonal to $\psi_{1}'$ is bounded below by some $\mu_1 > 0$. 

Motivated by this fact, consider the closed subspace of the functions $w \in C^{k,\alpha}_\e(\Gamma \times \R)$ satisfying the orthogonality condition
\begin{equation}\label{ortho}
\int_{\R}w(y,t)\dot{\psi}_{\varepsilon}(t)dt = 0, \qquad \mbox{for all} \qquad y \in \Gamma.
\end{equation}

One may use the strong positivity mentioned above and variational methods to construct solutions of the non-homogeneous problem $L_\e w = f$, where
\[L_{\varepsilon} = \varepsilon^{2}(\partial_{t}^{2} + \Delta_{\Gamma}) -W''(\psi_\e)\]
acting in functions defined in the weighted H\"{o}lder space $C^{k,\alpha}_{\varepsilon}(\Gamma\times\R)$ satisfying \eqref{ortho}. More precisely,

\begin{proposition}\cite[Propositions 3.1 and 3.2]{PA} There exists $\varepsilon_{0}>0$ such that, for all $\varepsilon\in (0,\varepsilon_{0})$ and for all $f\in C^{0,\alpha}_{\varepsilon}(\Gamma\times\R)$ satisfying the condition \eqref{ortho}, there exists a unique function $w\in C^{2,\alpha}_{\varepsilon}(\Gamma\times\R)$ which also satisfies \eqref{ortho}, and which solves
	\[L_{\varepsilon}w = f \qquad \mbox{in} \qquad \Gamma \times \R.\]
Moreover, there is a constant $C_1=C_1(n,g,W)>0$ such that
\begin{equation*}
	\|w\|_{C^{2,\alpha}_{\varepsilon}(\Gamma\times \R)} \leq C_1 \|f\|_{C^{0,\alpha}(\Gamma\times \R)}.
\end{equation*}
\end{proposition}

The last linear operator which appears in Pacard's construction is
\[\mathcal{L}_{\varepsilon} = \varepsilon^{2}\Delta_{g} -W''(\pm 1),\]
acting on functions on $M$ which lie in the following weighted H\"{o}lder space.

\begin{definition}
	For all $k \in \N$ and $\alpha \in (0,1)$, denote by $C^{k,\alpha}_{\varepsilon}(M)$ the space of functions $v \in C^{k,\alpha}(M)$ where the H\"{o}lder norm is calculated with respect to the scaled metric $\varepsilon^2\,g$, that is
		\[\|v\|_{C_\e^{k,\alpha}} := \sum_{j=0}^k \e^j \|\nabla^jv\|_{L^\infty} + \e^{k+\alpha}\sup_{x \neq y} \frac{|\nabla^k v(x) - \nabla^k v(y)|_g}{d(x,y)^\alpha}\]
	In the expression above, $\nabla$ denotes the Riemannian connection and $d$ denotes the Riemannian distance, both calculated with respect to the metric $g$ on $M$.
\end{definition}

\begin{proposition}\cite[Section 3.4]{PA} \label{oplt} For all $f\in C^{0,\alpha}_{\varepsilon}(M)$, there exists a unique function $w\in C^{2,\alpha}_{\varepsilon}(M)$ such that $\mathcal{L}_{\varepsilon}w = f$. Moreover, there exists a constant $C_2=C_2(n,g,W)>0$ such that
	\begin{equation*}
		\|w\|_{C^{2,\alpha}_{\varepsilon}(M)} \leq C_2 \|f\|_{C^{0,\alpha}_{\varepsilon}(M)}.
	\end{equation*}
\end{proposition}

\subsection{Approximate solutions}

	Based on the analogy between minimal hypersurfaces and solutions of \eqref{AC}, we aim to construct a solution as a perturbation of $\tilde{\psi_\e}(Z(y,z)) = \psi_\e(z)$, capped to $\pm 1$ outside a small neighborhood of $\Gamma$. Note again that this function has $\Gamma$ as its zero set, and it converges to $\pm 1$ \emph{exponentially} as $z \to \pm \infty$. 
	
Define $\bar \psi_\e:\mathcal{V} \to \R$ by
	\[\bar{\psi}_\e(Z(y,z))=\psi_\e(z).\]
We use cutoff functions to extend $\bar{\psi}_\e$ to $M$ in the sense described above: for small $\e>0$, consider a smooth function $\chi_k$, for $k=1,\ldots, 5$ such that
	\[ \chi_k(Z(y,z)) = \left\{ \begin{array}{rcl} 1, & \mbox{for} & |z| \leq \e^{\delta_*} \left(\frac{101-2k}{100}\right) \\[4pt]  0, & \mbox{for} & |z| \geq \e^{\delta_*}\left(\frac{102-2k}{100}\right) \end{array}  \right. \]
where $\delta_* \in (0,1)$ is fixed. We may assume $\chi_k$ is a smooth function defined on $M$ such that $\|\chi_k\|_{C_{\e^{\delta_*}}^{2,\alpha}(M)}$ is uniformly bounded (with respect to $\e$). 

We let
	\[\tilde{\psi_\e} = \chi_1 \bar\psi_\e \pm (1-\chi_1),\]
extended to $M$ as $\pm 1$ on the components $M_{\pm}$ of $M \setminus \Gamma$.

Given any function $\zeta \in C^{2,\alpha}(\Gamma)$ with small $C^0$ norm and small $\e>0$, we have an induced diffeomorphism $D_\zeta$ on $M$ defined by
	\[ D_\zeta(Z(y,z)) = Z(y,z-\chi_2(Z(y,z))\zeta(y)) \]
for $y \in \Gamma$ and $z$ such that $(y,z) \in \supp \chi_2$, and extended as $\id_M$ on the complement of $\supp \chi_2$. Finally, write 	
 \[\tilde u = u \circ D_\zeta^ {-1}.\]

We are looking for solutions of \eqref{AC} for which
	\[\tilde u = \tilde{\psi}_\e + v\]
for some function $v$ on $M$.
%\subsection{Rewriting the equation}
%Given a function $\zeta \in C^{2,\alpha}(\Gamma)$ with sufficiently small norm, consider $D_{\zeta}$ the diffeomorphism introduced above and write $u = \bar{u}\circ D_{\zeta}$. The equation
%\begin{equation*}
%	\varepsilon^{2} \Delta u + u - u^{3} = 0 \quad \mbox{ in } \quad M
%\end{equation*}
%can be rewritten as 
%\begin{equation*}
%	\varepsilon^{2} \Delta(\bar{u}\circ D_{\zeta})\circ D_{\zeta}^{-1} + \bar{u} - \bar{u}^{3} = 0
%\end{equation*}
%Now we look for a solution of this equation given as a perturbation of the 1-dimensional profile $\tilde{\psi}_{\varepsilon}$. We write
%\begin{equation*}
%	\bar{u} = \tilde{\psi}_{\varepsilon} + v,
%\end{equation*}
The Allen-Cahn equation can be rewritten as
\begin{equation} \label{eqv}
	\varepsilon^{2}\Delta(v\circ D_{\zeta})\circ D_{\zeta^{-1}} -W''(\tilde\psi_\e)v + P_{\varepsilon}(\zeta) - Q_{\varepsilon}(v) = 0,
\end{equation}
where 
\begin{equation*}
	P_{\varepsilon}(\zeta) := \varepsilon^{2} \Delta(\tilde{\psi}_{\varepsilon}\circ D_{\zeta})\circ D_{\zeta}^{-1} -W'(\tilde\psi_\e)
\end{equation*}
is the error corresponding to the approximate solution $\tilde{\psi}_{\varepsilon}$, and 
\begin{equation*}
	Q_{\varepsilon}(v) := W'(\tilde\psi_\e+v)-W'(\tilde\psi_\e) - W''(\tilde\psi_\e)v
\end{equation*}
collects the nonlinear terms of $v$. 

 We use the same technique employed by Pacard in \cite[Section 3.5]{PA} to decompose the perturbation $v$. If $v$ is small (in terms of $\e$), then $\e^2 \Delta u -W'(u)$ is approximately $\mathcal{L}_\e v-Q_\e(v)$ outside a neighborhood of $\Gamma$. Assume %%%%%%%%%%%%%%%%%%%%%%%%% OLHAR O DEL PINO
\begin{equation*}
	v : = \chi_{4} \bar v^{\sharp} + v^{\flat}
\end{equation*}
for functions $\bar v^\sharp: \mathcal{V} \to \R$ and $v^{\flat}:M \to \R$, where $v^\flat$ satisfies the following semilinear elliptic equation
\begin{align*}
	\mathcal{L}_{\varepsilon}v^{\flat} & = (\chi_{4}-1)\left[\varepsilon^{2}(\Delta(v^{\flat}\circ D_{\zeta})\circ D_{\zeta}^{-1} - \Delta v^{\flat}) - (W''(\tilde\psi_\e)-W''(\pm 1))v^\flat +\right.\\
	& \qquad \left. + P_{\varepsilon}(\zeta) - Q_{\varepsilon}(\chi_{4}\bar v^{\sharp} + v^{\flat})\right] - \varepsilon^{2}\left(\Delta((\chi_{4}v^{\sharp})\circ D_{\zeta}) - \chi_{4}\Delta(\bar v^{\sharp}\circ D_{\zeta})\right)\circ D_{\zeta}^{-1}.
\end{align*}
We use the Fermi coordinates $Z$ to write $\bar v^\sharp$ as a function defined on an open set in $\Gamma \times \R$ by writing $v^\sharp = Z^*\bar v^\sharp$. Let $N_{\varepsilon}(v^{\flat},v^{\sharp},\zeta)$ denote the right-hand side above, so that this equation reads
\begin{equation}\label{eq001}
	\mathcal{L}_{\varepsilon}v^{\flat} = N_{\varepsilon}(v^{\flat},v^{\sharp},\zeta)
\end{equation}

Observe that $N_{\varepsilon}(v^{\flat},v^{\sharp},\zeta)=Q_\e(v^\flat)$ outside $\supp \chi_1$, and it vanishes on $\{\chi_4=1\}\cap M$. This allows us to obtain improved estimates for a function $v^\flat$ satisfying this equation. In fact, by Proposition \ref{oplt}, if $w \in C^{2,\alpha}_{\varepsilon}(M)$ satisfies $\mathcal{L}_{\varepsilon}w = f$, then it holds 
\begin{equation*}
	\|w\|_{C^{2,\alpha}_{\varepsilon}(M)} \leq C_2\|f\|_{C^{0,\alpha}_{\varepsilon}(M)}.
\end{equation*}
If we assume $f\equiv 0$ in the support of $\chi_{4}$, then (cf. \cite[Remark 3.2]{PA}, see also \cite[Lemma 7.8]{OC}) it holds
\begin{equation*}
\|w\|_{\tilde{C}^{2,\alpha}_{\varepsilon}(M)} \leq \tilde C_2\|f\|_{C^{0,\alpha}_{\varepsilon}(M)}
\end{equation*}
where $\tilde C_2 = \tilde C_2 (n,g,W,\delta_*)>0$ and $\tilde C^{k,\alpha}_\e(M)$ denotes the modified H\"older norm:
\begin{equation}
	\|v\|_{\tilde{C}_{\varepsilon}^{k,\alpha}(M)} := \varepsilon^{-2}\|\chi_{5}v\|_{C^{2,\alpha}_{\varepsilon}(M)} + \|v\|_{C^{2,\alpha}_{\varepsilon}(M)}.
	\end{equation}

Combining \eqref{eqv} and \eqref{eq001}, we see that it to suffices find $v^\sharp$ which solves
\begin{align}\label{eqvsus}
	\varepsilon^{2}\Delta(\bar v^{\sharp}\circ D_{\zeta})\circ D_{\zeta}^{-1} -W''(\tilde\psi_\e) \bar v^{\sharp} &= -\varepsilon^{2}\left(\Delta(v^{\flat}\circ D_{\zeta})\circ D_{\zeta}^{-1} - \Delta_{g}v^{\flat}\right) +\nonumber\\
	&+ (W''(\tilde\psi_\e) - W''(\pm 1))v^{\flat}- P_{\varepsilon}(\zeta) + Q_{\varepsilon}(\chi_{4}\bar v^{\sharp} + v^{\flat}) 
\end{align} 
on $\supp \chi_{4}$. We may expand the left-hand side above using the expression for the Laplacian and its asymptotic expansion in terms of the height mentioned in \ref{fermi} to check that the first-order term in $\varepsilon^{2}\Delta(\bar v^{\sharp}\circ D_{\zeta})\circ D_{\zeta}^{-1}$ (with respect to $z$) can be written, in Fermi coordinates, as $L_\e v^\sharp - \e (J_\Gamma\zeta) \dot \psi_\e$. 

 Since we only need this equation to be satisfied on the support of $\chi_{4}$, it suffices to find $\zeta$ and $v^\sharp$ so that
\begin{align*}
	L_{\varepsilon}v^{\sharp} - \varepsilon J_{\Gamma}\zeta \dot{\psi}_{\varepsilon} & = (Z^*\chi_{3})\left[L_{\varepsilon}v^{\sharp} -\varepsilon^{2}Z^*[\Delta(\bar v^{\sharp}\circ D_{\zeta})\circ D_{\zeta}^{-1}] +W''(\tilde\psi_\e) v^{\sharp} \right.\\ 
	&\qquad \left. -\varepsilon^{2}Z^*\left(\Delta(v^{\flat}\circ D_{\zeta})\circ D_{\zeta}^{-1} - \Delta v^{\flat}\right) + (W''(\tilde\psi_\e)-W''(\pm 1))(Z^*v^{\flat}) \right.\\ 
	&\hspace{160pt} \left. - P_{\varepsilon}(\zeta) + Z^*Q_{\varepsilon}(\chi_{4}\bar v^{\sharp} + v^{\flat}) - \varepsilon J_{\Gamma}\zeta \dot{\psi}_{\varepsilon}\right]
\end{align*}
where we used the fact that $\tilde{\psi}_{\varepsilon} = \bar{\psi}_{\varepsilon}$ in the support of $\chi_{3}$, and we emphasize that $v^\sharp = Z^*\bar v^\sharp$. For short, the right side of the equation will be denoted by $M_{\varepsilon}(v^{\flat},v^{\sharp},\zeta)$ so this equation reads
\begin{align}\label{eq003}
	L_{\varepsilon}v^{\sharp} - \varepsilon J_{\Gamma}\zeta \dot{\psi}_{\varepsilon} = M_{\varepsilon}(v^{\flat},v^{\sharp},\zeta).
\end{align}
We split this equation by projecting it over the space of functions which satisfy \eqref{ortho}, and its orthogonal complement. More precisely, if we assume $v^\sharp$ satisfies \eqref{ortho}, and if we denote by $\Pi:C^{2,\alpha}_\e(\Gamma\times \R) \to C^{2,\alpha}(\Gamma)$ the orthogonal projection on $\dot{\psi}_{\varepsilon}$ and by $\Pi^\perp(f)=f-\Pi(f)\dot\psi_\e$ its orthogonal complement, then equation \eqref{eq003} may be rephrased as the system
\begin{align}
	\label{eq004}L_{\varepsilon}v^{\sharp} &= \Pi^{\perp}(M_{\varepsilon}(v^{\flat},v^{\sharp},\zeta))\\
\label{eq005}	-\varepsilon J_{\Gamma}\zeta &= \Pi(M_{\varepsilon}(v^{\flat},v^{\sharp},\zeta)).
\end{align}
%of the form 
%\begin{equation*}
%	u = \left(\tilde{u}_{\varepsilon} + \chi_{4}v^{\#} + v^{\flat}\right)\circ D_{\zeta}
%\end{equation*}
%where the function $v^{\#}$ is defined on $\Gamma\times \Re$, the function $v^{\flat}$ is defined in $M$ and the function $\zeta$ defined on $\Gamma$, and satisfy 
%\begin{eqnarray}
%\label{eq1}	\mathcal{L}_{\varepsilon}v^{\flat} &= &N_{\varepsilon}(v^{\flat},v^{\#},\zeta),\\
%	\label{eq2} L_{\varepsilon}v^{\#} &= &\Pi^{\perp}(M_{\varepsilon}(v^{\flat},v^{\#},\zeta)),\\
%	\label{eq3} -\varepsilon J_{\Gamma}\zeta& = &\Pi(M_{\varepsilon}(v^{\flat},v^{\#},\zeta)).
%\end{eqnarray}

However, due to the existence of nontrivial Jacobi fields, we can not directly follow the same strategy of \cite{PA} to solve \eqref{eq005}. In order to overcome the lack of injectivity of $J_\Gamma$, we consider a projected version of this problem, namely we look for $\zeta$ and certain constants $\{c_{i}\}_{i=1}^{J}$ such that
\begin{align}\label{eq006}
	-\varepsilon J_{\Gamma}\zeta = \Pi(M_{\varepsilon}(v^{\flat},v^{\sharp},\zeta)) + \e\sum_{i=1}^{J}c_{i}\hat{z}_{i}, \nonumber\\
	\int_{\Gamma} \hat{z}_{i} \zeta \dvol_{g} = 0, \quad i = 1,\dots, J,
\end{align}
where $\hat{z}_{1},\dots, \hat{z}_{J}$ is a basis of $\ker(J_\Gamma)$.

In Section \ref{exist} we show that there are functions $v^{\flat}$, $v^{\sharp}$ and $\zeta$ satisfying the system given by equations \eqref{eq001}, \eqref{eq004} and \eqref{eq006}. Consequently, the function ${u =  \left(\tilde{\psi}_{\varepsilon} + \chi_{4}\bar v^{\sharp} + v^{\flat}\right)\circ D_{\zeta}}$ satisfies 
\begin{equation}\label{eq007}
\varepsilon^{2} \Delta  u(x) - W'(u(x)) = \e\sum_{j=1}^{J}\dot{\psi}_{\varepsilon}(z-\zeta(y))\chi_{4}(Z(y,z-\zeta(y)))c_{j}\hat{z}_{j}(y).
\end{equation}
in the support of $\chi_{4}$, where $x = Z(y,z)$, and $u$ satisfies the Allen-Cahn equation \eqref{AC} on $M\backslash \supp \chi_{4}$. 

We can check equation \eqref{eq007} by noting that equations \eqref{eq004} and \eqref{eq006} imply
\begin{align*}
	L_{\varepsilon}v^{\sharp}(y,z-\zeta(y)) &- \varepsilon J_{\Gamma}(y)\zeta \dot{\psi}_{\varepsilon}(z-\zeta(y)) \\ &= M_{\varepsilon}(v^{\flat},v^{\sharp},\zeta)(y,z-\zeta(y)) + \varepsilon\sum_{i=1}^{J}c_{i}\hat{z}_{i}\dot{\psi}_{\varepsilon}(z-\zeta(y)),
\end{align*}
and, consequently,
\begin{align}\label{eq002}
\varepsilon^{2}\Delta (\bar v^{\sharp}\circ D_{\zeta})(Z(y,z)) &-W''({\psi}_{\varepsilon}^{2}(z-\zeta(y))) v^{\sharp}(y,z-\zeta(y))  \nonumber\\
& = -\varepsilon^{2}\left(\Delta (v^{\flat}\circ D_{\zeta})(Z(y,z)) - \Delta v^{\flat}(Z(y,z-\zeta(y)))\right) \nonumber \\
&\quad+ (W''({\psi}_{\varepsilon}^{2}(z-\zeta(y)))-W''(\pm 1))v^{\flat}(Z(y,z-\zeta(y))) - P_{\varepsilon}(\zeta)(y) \\ 
& \quad + Q_{\varepsilon}(\chi_{4}\bar v^{\sharp} + v^{\flat})(Z(y,z-\zeta(y))) + \varepsilon\sum_{i=1}^{J}c_{i}\hat{z}_{i}\dot{\psi}_{\varepsilon}(z-\zeta(y)) \nonumber
\end{align} 
in the support of $\chi_{4}$. Multiplying \eqref{eq002} by $\chi_{4}(Z(y,z-\zeta(y)))$ and adding the equation satisfied by $v^{\flat}$ at $Z(y,z-\zeta(y))$, namely \eqref{eq001}, we obtain \eqref{eq007}.

It remains to show that the constants $c_{i}$ vanish, so the function $u$ is actually a solution of the original equation \eqref{AC}, concluding the proof of Theorem \ref{main}. This will be tackled in Section \ref{conc}.

\subsection{The Jacobi operator}

In this section we discuss the linear problem of finding a function $\zeta$ such that, for certain constants $c_{1},\ldots,c_{J} \in \R$, 
\begin{align}\label{aux}
J_{\Gamma}\zeta & = f + \sum_{i=1}^{J}c_{i}\hat{z}_{i},\nonumber\\
\int_{\Gamma} \hat{z}_{i}\zeta \dvol_{g} & = 0, \qquad i=1,\ldots,J
\end{align}
where $\{\hat z_1\}_{i=1}^J$ is any $L^2$-orthonormal basis of $\ker(J_\Gamma)$. 

\begin{proposition} \label{jac} Given $f \in C^{0,\alpha}(\Gamma)$, there exists a unique bounded solution $W \in C^{2,\alpha}(\Gamma)$ of the problem \eqref{aux}. Moreover, there exists a positive constant $C_3=C_3(n,g,\Gamma) > 0$ such that the constants $c_i$ satisfy
\begin{equation*}
	|c_{i}| \leq C_3\|f\|_{L^2(\Gamma)},
\end{equation*}
and we have the following estimate
\begin{equation*}
\|W\|_{C^{2,\alpha}(\Gamma)} \leq C_3\|f\|_{C^{0,\alpha}(\Gamma)}.
\end{equation*}	
\end{proposition}

\begin{proof}
Let us consider the Hilbert space  $H$ of the functions $w \in W^{1,2}(\Gamma)$ satisfying the restriction
\begin{equation*}
	\int_{\Gamma}\hat{z}_{i}w \dvol_{g} = 0, \quad i=1,\ldots,J.
\end{equation*}	

In this space, problem \eqref{aux} can be formulated, in the weak form, as finding $W \in H$ such that  
\begin{equation} \label{weak1}
	B_H[w,\psi] = \int_{\Gamma} f\,\psi  \quad \mbox{ for all } \quad \psi \in H,
\end{equation}
where
	\[B_H[w,\psi]=\int_{\Gamma}\langle \nabla w, \nabla \psi\rangle + (|A_\Gamma|^{2} + \Ric(\nu,\nu))w\psi.\]
Equivalently, $w$ is a weak solution if, and only if
	\begin{equation} \label{weak2}
	\int_{\Gamma}\langle \nabla w, \nabla \phi\rangle + (|A_\Gamma|^{2} + \Ric(\nu,\nu))w\phi = \int_{\Gamma} \tilde f\,\phi  \quad \mbox{ for all } \quad \phi \in W^{1,2}(\Gamma),
\end{equation}
where $\tilde f = f - \sum_{j=1}^J \int_\Gamma f \hat z_i$ is the projection of $f$ on $H$. In particular, if $f$ is of class $C^{0,\alpha}(\Gamma)$ then so is $\tilde f$, and by elliptic regularity we see that any weak solution is in $C^{2,\alpha}(\Gamma)$.

Suppose that $w \in H$ is a weak solution of this problem for $f = 0$. Given $\phi \in W^{1,2}(\Gamma)$, we may write $\phi = \psi + \sum_{i=1}^{J}a_{i}\hat z_{i}$, where $\psi \in H$ and $a_i = \int_\Gamma \phi \hat z_i$. Then
	\begin{align*}
		\int_{\Gamma}\langle \nabla w, \nabla \phi\rangle + (|A_\Gamma|^{2} + \Ric(\nu,\nu)) w \phi & = \int_{\Gamma}\langle \nabla w, \nabla \psi\rangle + (|A_\Gamma|^{2} + \Ric(\nu,\nu)) w \psi \\
		& \quad + \sum_{i=1}^J a_i \int_{\Gamma}\langle \nabla w, \nabla \hat z_i \rangle + (|A_\Gamma|^{2} + \Ric(\nu,\nu)) w \hat z_i\\
		& = -\sum_{i=1}^J a_i \int_{\Gamma} wJ_\Gamma \hat z_i = 0,
	\end{align*}
that is, $w$ is a weak solution of the Jacobi equation $J_\Gamma w =0$. By elliptic regularity, we see that $w$ is a classical solution, so $w \in \ker(J_\Gamma)$ and we get $w=0$, since $w \in H$.

The existence of weak solutions of \eqref{weak1} now follows from a standard argument. 
Finally, if $w$ solves \eqref{aux}, then for each $j=1,\ldots, J$
	\[0 = \int_{\Gamma} \hat z_j J_\Gamma w = \int_{\Gamma} f\, z_j + \sum_{i=1}^J c_j \int_\Gamma \hat z_i \hat z_j, \]
that is $c_j = -\int f\, \hat z_j$, so 
	\[|c_j| \leq \|\hat z_j\|_{L^2(\Gamma)} \|f\|_{L^2(\Gamma)}\leq C \| f \|_{L^2(\Gamma)}.\]
The estimate on the $C^{2,\alpha}$ norm of $u$ is a consequence of the standard Schauder estimates.
\end{proof}

\begin{remark}
In the next section, we apply Proposition \ref{jac} to solve \eqref{aux} with $f=-\e^{-1}\Pi(M_\e(v^\flat,v^\sharp,\zeta))$. We emphasize that, in this case, the constants $c_1,\ldots, c_J$ depend on the initial data $(v^\flat,v^\sharp,\zeta)$ and $\e$.
\end{remark}

\section{Proof of Theorem \texorpdfstring{\ref{main}}{1.1}} \label{proof}

\subsection{A fixed-point problem} \label{exist}
We formulate the problem of finding a solution of \eqref{AC} of the form $u = (\tilde \psi_\e + \chi_4 \bar v^\sharp + v^\flat) \circ D_{\zeta}$ as a fixed-point problem using the reduction described in the previous section. For this purpose, it suffices to find $v^\flat$, $v^\sharp$, and $\zeta$ satisfying the nonlinear coupled system \eqref{eq001}, \eqref{eq004}, and \eqref{eq005}.

We recall here the following result, which is proved using the local asymptotics of Fermi coordinates (see also \cite[Lemma 7.9]{OC}).

\begin{lemma} \label{estzero} \cite[Lemma 3.8]{PA} For any $\alpha \in (0,1)$ and any $\e>0$, we have
	\[\|N_\e(0,0,0)\|_{C^{0,\alpha}_\e(M)} + \|\Pi^\perp(M_\e(0,0,0))\|_{C^{0,\alpha}_\e(\Gamma \times \R)} + \e^{-1}\|\Pi(M_\e(0,0,0))_{C^{0,\alpha}(\Gamma)} \leq C_0 \e^2,  \]
where $C_0=C_0(n,g,W,\Gamma,\delta_*,\alpha)>0$.
\end{lemma}
	
The following estimates from \cite{PA} are the key result in the fixed-point argument. See also \cite[Lemma 7.10]{OC} for the detailed computations (in a slightly different context). 

\begin{lemma}\label{PAest} \cite[Lemma 3.9]{PA} Given $\bar C>0$ and $\alpha \in (0,1/4)$, there exist 
	\begin{align*}
		\delta & = \delta(g,W,\Gamma,\delta_*,\bar C)>0\\
		\e_0 & = \e_0(g,W,\Gamma,\delta_*,\bar C)>0 \\
		\tilde C & = \tilde C(g,W,\Gamma,\delta_*,\bar C)>0
	\end{align*}
such that if
		\begin{equation*}
		\|v^{\flat}\|_{\tilde{C}^{2,\alpha}_{\varepsilon}(M)} + \|v^{\sharp}\|_{C_\e^{2,\alpha}(\Gamma\times \R)} + \varepsilon^{2\alpha}\|\zeta\|_{C^{2,\alpha}(\Gamma)} \leq \bar{C}\varepsilon^{2}.
	\end{equation*}
and $\e\in(0,\e_0)$ then following estimates hold:
	\begin{align*}
	&\|N_{\varepsilon}(v_{2}^{\flat},v_{2}^{\sharp},\zeta_{2}) - N_{\varepsilon}(v_{1}^{\flat},v_{1}^{\sharp},\zeta_{1})\|_{C^{0,\alpha}_{\varepsilon}(M)} \\ &\qquad \qquad\leq \tilde C \varepsilon^{\delta}\left(\|v_{2}^{\flat} - v_{1}^{\flat}\|_{C^{2,\alpha}_{\varepsilon}(M)} + \|v_{2}^{\sharp} - v_{1}^{\sharp}\|_{C^{2,\alpha}_{\varepsilon}(\Gamma\times \R)} + \|\zeta_{2}-\zeta_{1}\|_{C^{2,\alpha}(\Gamma)}\right)
	\end{align*}
	\begin{align*}
	&\|\Pi^{\perp}(M_{\varepsilon}(v_{2}^{\flat},v_{2}^{\sharp},\zeta_{2}) - M_{\varepsilon}(v_{1}^{\flat},v_{1}^{\sharp},\zeta_{1}))\|_{C^{0,\alpha}_{\varepsilon}(\Gamma\times\R)} \\ &\qquad \qquad\leq \tilde C \varepsilon^{\delta}\left(\|v_{2}^{\flat} - v_{1}^{\flat}\|_{\tilde{C}^{2,\alpha}_{\varepsilon}(M)} + \|v_{2}^{\sharp} - v_{1}^{\sharp}\|_{C^{2,\alpha}_{\varepsilon}(\Gamma\times \R)} + \|\zeta_{2}-\zeta_{1}\|_{C^{2,\alpha}(\Gamma)}\right)
	\end{align*}
	and 
	\begin{align*}
	&\|\Pi(M_{\varepsilon}(v_{2}^{\flat},v_{2}^{\sharp},\zeta_{2}) - M_{\varepsilon}(v_{1}^{\flat},v_{1}^{\sharp},\zeta_{1}))\|_{C^{0,\alpha}(\Gamma)} \\ & \qquad\qquad \leq \tilde C\varepsilon^{1-\alpha}\|v_{2}^{\sharp} - v_{1}^{\sharp}\|_{C^{2,\alpha}_{\varepsilon}(\Gamma\times\R)} + \tilde C\varepsilon^{1+\delta}\left(\|v_{2}^{\flat} - v_{1}^{\flat}\|_{\tilde{C}^{2,\alpha}_{\varepsilon}(M)} + \|\zeta_{2} - \zeta_{1}\|_{C^{2,\alpha}(\Gamma)} \right),
	\end{align*}
	\qed
	\end{lemma}

For $k=1,2$, pick $v_{k}^{\flat} \in C^{2,\alpha}_{\varepsilon}(M)$, $v_{k}^{\sharp} \in C^{2,\alpha}_{\varepsilon}(\Gamma\times\R)$ and $\zeta_{k}\in C^{2,\alpha}(\Gamma)$ such that
	\begin{equation*}
\|v^{\flat}_k\|_{\tilde{C}^{2,\alpha}_{\varepsilon}(M)} + \|v^{\sharp}_k\|_{C^{2,\alpha}(\Gamma\times \R)} + \varepsilon^{2\alpha}\|\zeta_k\|_{C^{2,\alpha}(\Gamma)} \leq \bar{C}\varepsilon^{2}.
\end{equation*}
Suppose that $c_{i,k}$ and $W_k$ are the solutions of problem \eqref{aux} given by Proposition \ref{jac}, for ${f=-\e^{-1}\Pi(M_{\varepsilon}(v^{\flat}_{k},v^{\sharp}_{k},\zeta_{k}))}$, so that
\begin{equation*}
	-\varepsilon J_{\Gamma}W_{k} = \Pi(M_{\varepsilon}(v^{\flat}_{k},v^{\sharp}_{k},\zeta_{k})) + \e\sum_{i=1}^{J}c_{i,k}\hat{z}_{i}.
\end{equation*}
By Proposition \ref{jac} and Lemma \ref{PAest}, we have
\begin{align}
|c_{i, 1} - c_{i, 2}| &\leq C_3 \|\e^{-1}\Pi(M_{\varepsilon}(v_{1}^{\flat},v_{1}^{\sharp},\zeta_{1})- M_{\varepsilon}(v_{2}^{\flat},v_{2}^{\sharp},\zeta_{2}))\|_{C^{0,\alpha}(\Gamma)} \label{difci}\\
&\leq  C_3\cdot \tilde C \cdot \left(\varepsilon^{-\alpha}\|v_{2}^{\sharp} - v_{1}^{\sharp}\|_{C^{2,\alpha}_{\varepsilon}(\Gamma\times\R)} + \varepsilon^{\delta}\left(\|v_{2}^{\flat} - v_{1}^{\flat}\|_{\tilde{C}^{2,\alpha}_{\varepsilon}(M)} + \|\zeta_{2} - \zeta_{1}\|_{C^{2,\alpha}(\Gamma)} \right)\right) \nonumber
\end{align}

We may now apply a fixed point argument to obtain solutions of \eqref{eq007}. We follow the notation in \cite{OC}. Denote
	\begin{align*}
		\mathcal{U}(\e, \bar C) & = \left\{ (v^\flat,v^\sharp,\zeta) \in \tilde C^{2,\alpha}_\e(M) \times C^{2,\alpha}_\e(\Gamma \times \R) \times C^{2,\alpha}(\Gamma): \|(v^\flat,v^\sharp,\zeta)\|_{\mathcal{U}} \leq \bar C \e^2 \right\}
	\end{align*}
where
	\[\|(v^\flat,v^\sharp,\zeta)\|_{\mathcal{U}} =  \|v^\flat\|_{\tilde C^{2,\alpha}_\e(M)} + \|v^\sharp\|_{C^{2,\alpha}_\e(\Gamma \times \R)} + \e^{2\alpha}\|\zeta\|_{C^{2,\alpha}(\Gamma)}.\]
Then, for any $(v^\flat,v^\sharp,\zeta) \in \mathcal{U}(\e,\bar C)$, using Lemmas \ref{estzero} and \ref{PAest}, Proposition \ref{jac} and \eqref{difci}, we obtain
	\begin{align*}
		\|N_\e(v^\flat,v^\sharp,\zeta)\|_{C^{0,\alpha}_\e(M)} & \leq \bar C \e^2 + \tilde C'\e^{2-2\alpha+\delta}\\
		\|\Pi^\perp(M_\e(v^\flat,v^\sharp,\zeta))\|_{C^{0,\alpha}_\e(\Gamma\times \R)} & \leq \bar C \e^2 + \tilde C' \e^{2-2\alpha+\delta}\\
		\|\e^{-1}\Pi(M_\e(v^\flat,v^\sharp,\zeta))\|_{C^{0,\alpha}(\Gamma)} & \leq \bar C \e^2 + \tilde C'(\e^{2-2\alpha+\delta} + \e^{2-\alpha})\\
		\|(c_1,\ldots, c_J)\|_\infty & \leq C_3\bar C \e^2 + C_3\tilde C'(\e^{2-2\alpha+\delta} + \e^{2-\alpha})
	\end{align*}
where $\tilde C'=3\cdot\bar C\cdot\tilde C$, and $\tilde C$ is given by Lemma \ref{PAest}. For large $\bar C>1$, small $\e>0$ and small $\alpha$ (depending on $\delta$), consider the \emph{solution map}
	\[\Phi_\e : \mathcal{U}(\e,\bar C) \to \mathcal{U}(\e,\bar C)\]
given by $\Phi_\e(v^\flat,v^\sharp, \zeta) = (V^\flat, V^\sharp, W)$, where
	\begin{align*}
		\mathcal{L}_\e V^\flat & = N_\e(v^\flat, v^\sharp, \zeta)\\[2pt]
		L_\e V^\sharp & = \Pi^\perp(M_\e(v^\flat,v^\sharp, \zeta))\\
		-J_\Gamma W & = \e^{-1}\Pi(M_\e(v^\flat,v^\sharp, \zeta)) + \sum_{i=1}^J c_j(v^\flat,v^\sharp,\zeta) \hat z_i
	\end{align*}
which is well-defined by the results of Section \ref{linear}. We remark again that the constants $c_j$ depend on $v^\beta$, $v^\sharp$ and $\zeta$. By the estimates above, this map is a contraction on $\mathcal{U}(\e,\bar C)$, endowed with the metric induced by $\|\cdot\|_{\mathcal{U}}$, provided $\e>0$ is sufficiently small. Therefore, we conclude that $\Phi$ has a fixed point $(v^\flat,v^\sharp,\zeta) \in \mathcal{U}(\e,\bar C)$.

\subsection{Existence of solutions of \texorpdfstring{\eqref{AC}}{(1.1)}} \label{conc}

Summarizing, we have found a solution $u$ of equation \eqref{eq007}, namely
\begin{equation*}
\varepsilon^{2} \Delta u(x) -W'(u(x)) = \varepsilon\sum_{i=1}^{J}c_{i}\hat{z}_{i}(y)\dot{\psi}_{\varepsilon}(z-\zeta(y))\chi_{4}(Z(y,z-\zeta(y))).
\end{equation*}
for certain constants $c_1,\ldots, c_J$, where $u(x) = (\tilde{\psi}_{\varepsilon} + \chi_{4}\bar v^{\sharp} + v^{\flat})\circ D_{\zeta}(x)$, and $x = Z(y,z)$.

By our hypothesis on $\Gamma$, there are Killing fields $\{X_{i}\}_{i=1}^J$ such that $\{z_{i} = \langle X_{i},\nu\rangle\}_{i=1}^J$ forms a basis for $\ker(J_\Gamma)$. Thus, we may assume
\begin{equation*}
{z}_{i}=\sum_{l=1}^{J}\beta_{il}\hat z_{l}, \qquad i=1,...,J,
\end{equation*}
where $\{\hat z_l\}_{l=1}^J$ is a fixed $L^2$-orthonormal basis of $\ker(J_\Gamma)$, that is,
\begin{equation*}
\int_{\Gamma} \hat{z}_{i}\hat{z}_{j} = \delta_{ij}, \qquad i,j=1,...,J,
\end{equation*}
for an invertible matrix $B=(\beta_{i\,l})$ with constant coefficients depending only on $\Gamma$.

Our goal is to prove that the constants $c_1,\ldots, c_J$ vanish. First, we observe that
\begin{equation} \label{eq009}
\int_{M}(\varepsilon^{2}\Delta u -W'(u))Y_{i} = 0 \quad \mbox{ for all } \quad i= 1,...,J.\\[3pt]
\end{equation}
This follows from the fact that the vector fields $X_i$ preserve the energy of $u$, and can be checked by the following computation.

\begin{proposition} \label{integral} Let $F \in C^1(\R)$ and let $u \in C^2(M)$. If $Y$ is a Killing field on $M$. Then
	\[\int_M (-\Delta u + F'(u))\left\langle Y,\nabla u\right\rangle = 0.\]
\end{proposition}

\begin{proof}
	Let $e=|\nabla u|^2/2 + F(u)$. Since
		\[\langle X, \nabla e \rangle = \nabla u(\langle X, \nabla u \rangle) + F'(u)\left\langle X,\nabla u\right\rangle - \left\langle \nabla_{\nabla u}X, \nabla u \right\rangle,\]
	for any $C^1$ vector field $X$ on $M$, it follows from the divergence theorem
		\begin{equation*} 
			\int_M (-\Delta u + F'(u)) \langle X, \nabla u \rangle = \int_M \left[ \left\langle \nabla_{\nabla u}X, \nabla u \right\rangle - e_\e \Div X \right].
		\end{equation*}
	If $X=Y$ is a Killing field, then both terms in the right-hand side above vanish. This concludes the proof.\qedhere
	
	%Let $\phi^t$ be the flow generated by $X$ on $M$. Since $\psi^t$ is an isometry, its Jacobian determinant satisfies $$|J\psi^t|= \sqrt{\det( (\psi^t)^*g)}=1$$ and $\phi^t$ preserves the volume form. Hence, by the change of variables formula,
	%\[\int_M (v \circ \psi^t)\, \dvol_g = \int_M v\,\dvol_g.\]
	%Differentiating both sides with respect to $t$, and evaluating at $t=0$ gives
	%\[\int_M \left\langle X, \nabla v\right\rangle\, \dvol_g = \frac{d}{dt}\bigg|_{t=0} \int_M (v \circ \psi^t)\, \dvol_g = 0. \]
	%On the other hand,
	%\[\left\langle X, \nabla v\right\rangle = \left\langle \nabla_{\nabla u} \nabla u, X\right\rangle + F'(u)\left\langle X, \nabla u\right\rangle = \nabla u(\left\langle \nabla u, X \right\rangle)+F'(u)\left\langle X,\nabla u\right\rangle \]
	%where we used the fact that $X$ is a Killing field, so $\left\langle \nabla u, \nabla_{\nabla u} X \right\rangle = 0$. Therefore, by the divergence Theorem, we get
	%\[\int_M \left\langle X, \nabla v \right\rangle\,\dvol_g = \int_M(-\Delta u + F'(u))\left\langle \nabla u, X \right\rangle\,\dvol_g. \]
\end{proof}

By \eqref{eq009} and \eqref{eq007}, we find that
\begin{equation}\label{eqci2}
	\sum_{j=1}^{J}\varepsilon c_j\int_{M} \dot{\psi}_{\varepsilon}(z-\zeta(y))\chi_{4}(Z(y,z-\zeta(y)))\hat{z}_{j}(y){Y}_i(x) \dvol_{g}(x) = 0
\end{equation}
for each $i=1,\ldots, J$.

In the next lemma we express $Y_i$ in terms of the one-dimensional solutions of \eqref{AC}, plus a small error term. We will use the fact that the solution $(v^\flat, v^{\sharp}, \zeta)$ of the fixed point problem obtained in the previous section satisfies
\begin{equation} \label{estsol}
\|v^{\flat}\|_{\tilde{C}^{2,\alpha}_{\varepsilon}(M)} + \|v^{\sharp}\|_{C_\e^{2,\alpha}(\Gamma\times \R)} + \varepsilon^{2\alpha}\|\zeta\|_{C^{2,\alpha}(\Gamma)} \leq \bar{C}\varepsilon^{2}.
\end{equation}

\begin{lemma} \label{grad}
In Fermi coordinates near the minimal hypersurface $\Gamma$, the functions $Y_{i} = \langle X_{i},\nabla u\rangle$ for $i=1,\ldots, J$ can be written as 
	\begin{equation*}
		Y_{i}(x) = \frac{1}{\varepsilon}\dot{\psi}_{\varepsilon}\left(z-\zeta(y)\right)z_i(y) + R_\e(x),
	\end{equation*}
where $x=Z(y,z)$ and
	\[\|R_\e\|_{L^\infty} \leq C \e\]
on the support of $\chi_4$.
\end{lemma}

\begin{proof}
Since $\tilde \psi_\e = \psi_\e$ on $\supp\chi_4$, for any $x = Z(y,z) \in \supp\chi_{4}$ it holds
	\begin{eqnarray*}
		u(x) &=& (\tilde{\psi_\e} + \chi_{4}\bar v^{\sharp} + v^{\flat})\circ D_{\zeta}(Z(y,z))\\ 
		& =& \psi_{\varepsilon}(z-\zeta(y)) + (\chi_{4}\bar v^{\sharp} + v^{\flat})\circ D_{\zeta}(x)
	\end{eqnarray*}
If we let $p_\zeta:\mathcal{V} \to \R$ be $p_\zeta(Z(y,z))=z-\zeta(y)$ and write $w^\sharp = \chi_4 \bar v^\sharp \circ D_\zeta$ and $w^\flat = v^\flat \circ D_\zeta$, then we get
	\begin{equation*}
	\nabla u(x) = \frac{1}{\varepsilon}\dot{\psi}_{\varepsilon}(z-\zeta(y))\nabla p_\zeta(x) + \nabla w^\sharp(x) + \nabla w^\flat(x),
	\end{equation*}
	
First, we compute
	\[\left\langle X_i, \nabla p_\zeta \right\rangle = \left\langle X_i, \partial_z \right\rangle + R,\]
where
	\[\|R\|_{L^\infty}\leq C\|X_i\|_{L^\infty(M)} \|\nabla \zeta\|_{L^\infty(\Gamma)} \leq C \e^{2(1-\alpha)}. \]
and we used \eqref{estsol}. Note also that, since $X_i$ is a Killing field and $\nabla_{\partial_z}\partial_z=0$, we have
	\begin{align*}
		\left\langle X_i, \partial_z \right\rangle(Z(y,z)) - z_i(y) & = \int_0^1 \frac{d}{dt}\left\langle X_i(Z(y,tz)), \partial_z \right\rangle \, dt\\
			& = \int_0^1 \left\langle \nabla_{\partial_z}X_i,\partial_z \right\rangle + \left\langle X_i, \nabla_{\partial_z}\partial_z \right\rangle\,dt=0.
	\end{align*}
	
Next, observe that
	\[\nabla w^\sharp(x) = v^\sharp(D_\zeta(x))\cdot T_\zeta(x)(\nabla \chi_4(D_\zeta(x))) + \chi_4(D_\zeta(x)) \cdot T_\zeta(x)(\nabla v^\sharp(D_\zeta(x)))\]
where $T_\zeta$ is the transpose of the linear map $D(D_\zeta)_x: T_xM \to T_{D_\zeta(x)}M$. If we compute $D_\zeta$ and its derivative in Fermi coordinates we see that
	\[\|T_{\zeta}(x)\| \leq C(1+\|\nabla\zeta\|_{L^\infty}) \leq C(1+\e^{2(1-\alpha)}). \]
Along with \eqref{estsol}, this implies
	\begin{align*}
		\|\nabla w^\sharp \|_{L^\infty} & \leq \|v^\sharp\|_{L^\infty} \|T_\zeta\|_{L^\infty} \| \nabla \chi_4 \|_{L^\infty} + \|T_\zeta\|_{L^\infty} \|\nabla v^\sharp\|_{L^\infty}\\
			& \leq C\e^{2-\delta_*} + C\e \leq C\e
	\end{align*}
Finally, \eqref{estsol} and the estimates on the norm of $T_\zeta$ yield 
\begin{equation*}
\|\nabla w^\flat \|_{L^\infty}\leq C \e.\qedhere
\end{equation*} 
\end{proof}
We can use this lemma to express each integral in the left-hand side of \eqref{eqci2} in terms of the one-dimensional solution and the Jacobi fields only. We have
\begin{align}\label{int}
& \int_{M} \dot{\psi}_{\varepsilon}(z-\zeta(y))\chi_{4}(Z(y,z-\zeta(y)))\hat{z}_{j}(y)Y_{i}(x) \dvol_{g}(x)  \\  
&\qquad	=\int_{M}\frac{1}{\e}\left( \dot{\psi}_{\varepsilon}(z-\zeta(y))\right)^2\chi_{4}(Z(y,z-\zeta(y)))\hat{z}_{j}(y)z_{i}(y) \dvol_{g}(x) + o(1) \nonumber
\end{align}
%This assumptions will imply that 
%\begin{align*}
%\int_{M}\dot{\psi}_{\varepsilon}(z-\zeta(y))\chi_{4}(Z(y,z-\zeta(y)))\hat{z}_{j}(y)\langle\nabla((\chi_{4}v^{\sharp}+v^{\flat})\circ D_{\zeta}), X_{i}\rangle(x) = o(1)
%\end{align*}
%and
%\begin{align*}
%\frac{1}{\varepsilon} \dot{\psi}_{\varepsilon}^{2}(z-\zeta(y))\chi_{4}(Z(y,z-\zeta(y)))\hat{z}_{j}(y)\langle \nabla\zeta(y),X_i\rangle = o(1)
%\end{align*}
%as $\varepsilon \downarrow 0$. 
as $\e \downarrow 0$. By the change of variables formula, this integral can be written as
	\begin{align*}
		&\int_{M}\frac{1}{\e} \left(\dot{\psi}_{\varepsilon}(z-\zeta(y))\right)^2\chi_{4}(Z(y,z-\zeta(y)))\hat{z}_{j}(y)z_{i}(y) \dvol_{g}(x)\\
			 &\qquad = \int_\R \left[\int_\Gamma\frac{1}{\e}\left( \dot{\psi}_{\varepsilon}(z)\right)^2\bar\chi_{4}(z)\hat{z}_{j}(y)z_{i}(y) |J( D_\zeta^{-1}\circ Z)| (y,z)d\Gamma(y)\right]\,dz
	\end{align*}
where $\bar\chi_4(z)=\chi_4(Z(y,z))$. Since we have the rough estimate
	\[|J( D_\zeta^{-1}\circ Z)| = 1 + o(\e^{\delta_*}) \qquad \mbox{on} \qquad \supp \chi_4,\] %%%%%%%%%%%%%%%%%%%%%%%%%%%%%%%%%%%%% CHECAR CHECAR CHECAR CHECAR CHECAR
using Remark \ref{rem} we conclude
		\begin{align*}
		&\int_{M}\frac{1}{\e}\left( \dot{\psi}_{\varepsilon}(z-\zeta(y))\right)^2\chi_{4}(Z(y,z-\zeta(y)))\hat{z}_{j}(y)z_{i}(y) \dvol_{g}(x)\\
			 &\qquad = \int_\R\frac{1}{\e}\left( \dot{\psi}_{\varepsilon}(z)\right)^2\bar\chi_{4}(z)\int_\Gamma\left[\hat{z}_{j}(y)z_{i}(y) d\Gamma(y)\right]\,dz + o(1)\\
			 &\qquad = \sum_{l=1}^J\beta_{il}\int_\R\frac{1}{\e} \left(\dot{\psi}_{\varepsilon}(z)\right)^2\bar\chi_{4}(z) \int_\Gamma\left[\hat{z}_{j}(y)\hat z_{l}(y) d\Gamma(y)\right]\,dz + o(1)\\
			& = \sigma \beta_{ij} + o(1)
	\end{align*}
as $\varepsilon\downarrow 0$. We conclude that \eqref{eqci2} can be rephrased as
	\[\sum_{j=1}^J\e(\sigma \beta_{ij} + o(1))c_j=0 \qquad \mbox{for} \qquad i=1,\ldots, J.\]
Since $B=(\beta_{ij})$ is invertible, it follows that the constants $c_{i}$ vanish for sufficiently small $\e$, and $u$ solves \eqref{AC} on $M$.

\subsection{Index and nullity}
We conclude the proof of Theorem \ref{main} by proving that the Morse index and the nullity of the solution $u_\e=u$ coincides with the Morse index and the nullity of $\Gamma$, respectively, provided $\e$ is sufficiently small. 

It follows from the results of \cite{GasparInner} (see also \cite{Hiesmayr}) that 
	\begin{equation*}
		\ind(\Gamma) \leq m(u_\e), \quad \mbox{for sufficiently small $\e$}.
	\end{equation*}
We remark that the hypothesis that the Morse index of the solutions is uniformly bounded in \cite[Theorem A]{GasparInner} is required only in order to ensure that the limit interface $\Gamma$ is a minimal hypersurface with optimal regularity. In our case, we can directly check, e.g. using the description of $\nabla u_\e$ given in the proof of Lemma \ref{grad}, that the level sets of $u$ converge to $\Gamma$ in the sense of \emph{varifolds}. More precisely, it holds
	\begin{equation} \label{varcon}
		\lim_{\e \to 0^+} \int_{M \cap \{\nabla u_\e \neq 0 \}} \frac{\e|\nabla u_\e|^2}{2} \phi(x,T_x\{u_\e=u_\e(x)\})\,d\vol_g(x) = \int_\Gamma \phi(x,T_x \Gamma)\,d\Gamma(x)
	\end{equation}
for any continuous function $\phi$ on the Grassmannian bundle $G_n(M)$ of $n$-planes on $M$. See \cite[Section 2.2]{HutchinsonTonegawa}, \cite[Section 3]{Guaraco} or \cite[Section 2]{GasparInner} for further details. 

A direct computation (see Proposition \ref{integral}) shows that the functions $Y_i = \langle X_i,\nabla u_\e\rangle$, for $i=1,\ldots, J$, satisfy
	\[\delta^2 E_\e(Y_i) = \frac{d^2}{dt^2}\bigg|_{t=0} E_\e(u_\e + tY_i) = 0.\]
Moreover, these functions are linearly independent for sufficiently small $\e$ (depending on $g$, $\Gamma$ and $W$). In fact, assume that we can find a sequence $\e_k$ with $\e_k \downarrow 0$, and a sequence $a_k \in S^{J-1}$ such that $\phi_k=a_k \cdot (Y_1,\ldots, Y_J)$ vanishes identically, where we use the notation
	\[v \cdot (Y_1,\ldots, Y_J) = v^1 Y_1 + \ldots v^J Y_J, \quad \mbox{for} \quad v=(v^1,\ldots, v^J) \in S^{J-1}.\]
By passing to a subsequence, we may assume $a_k \to \alpha = (\alpha^1,\ldots, \alpha^J) \in S^{J-1}$. By \eqref{varcon} (see \cite[Proposition 2.2]{GasparInner}) the vector field $X=\alpha^1 X_1 + \ldots + \alpha^JX_J$ satisfies
	\begin{align*}
		0 & = \lim_{k \to \infty} \e_k\int_M\phi_k^2 \, d\mathcal{H}^{n+1} = 2\sigma\int_\Gamma \langle X, \nu \rangle^2 \, d\mathcal{H}^n.
	\end{align*}
This means $\langle X, \nu \rangle= \alpha^1 z_1 + \ldots + \alpha^J z_J$ vanishes on $\Gamma$, which contradicts our choice of $z_1,\ldots, z_J$. Therefore,
	\[\Null(\Gamma)=J\leq n(u_\e).\]

Conversely, the convergence result stated in \eqref{varcon} ensures that $\Gamma$ has \emph{multiplicity one}, as required in Section 5 of \cite{OC}. Therefore, by \cite[Theorem 5.11]{OC}, we get
	\[m(u_\e) + n(u_\e) \leq \ind(\Gamma) + \Null(\Gamma),\]
again, for sufficiently small $\e>0$. This implies $n(u_\e) = \Null(\Gamma)$ and $m(u_\e) = \ind(\Gamma)$ and concludes the proof of Theorem \ref{main}.

%%%%%%%%%%%%%%%%%%%%%%%%%%%%%%%%%%%%%%%%%%%%%%%%%%%%%%%%%%%%%%%%%%%%%%%%%%%%%%%%%%%%%%%%%%%%%%%%%%%%%%%%%%%%%%%%%%%%%%%%%%%%%%%%%%%%%%%%%%%%%%%%%%%%%%%
%%%%%%%%%%%%%%%%%%%%%%%%%%%%%%%%%%%%%%%%%%%%%%%%%%%%%%%%%%%%%%%%%%%%%%%%%%%%%%%%%%%%%%%%%%%%%%%%%%%%%%%%%%%%%%%%%%%%%%%%%%%%%%%%%%%%%%%%%%%%%%%%%%%%%%%
%%%%%%%%%%%%%%%%%%%%%%%%%%%%%%%%%%%%%%%%%%%%%%%%%%%%%%%%%%%%%%%%%%%%%%%%%%%%%%%%%%%%%%%%%%%%%%%%%%%%%%%%%%%%%%%%%%%%%%%%%%%%%%%%%%%%%%%%%%%%%%%%%%%%%%%

\section{Applications and concluding remarks} \label{applications}

\subsection{Examples and topology of nodal sets} There are large classes of minimal hypersurfaces in compact manifolds of positive curvature for which Theorem \ref{main} can be directly applied. We name a few of these examples in this section.

For each pair of positive integers $p$ and $q$ such that $q=n-p$, the product of spheres
	\[M_{p,q}= S^p(\sqrt{p/n}) \times S^q(\sqrt{q/n})\]
is a minimal hypersurface in $S^{n+1}$. In particular, for $n=2$ and $p=q=1$, we get precisely the Clifford torus in $S^3$. It was noted by Lawson and Hsiang in \cite{LawsonHsiang} that all Jacobi fields of $M_{p,q}$ are Killing-Jacobi fields. Since any hypersurface in $S^{n+1}$ is separating, we see that $M_{p,q} \subset S^{n+1}$ satisfies the hypotheses of Theorem \ref{main}. 

Note further that 
	\[M_{2p-1,2q-1} \subset S^{2n-1} \quad \mbox{and} \quad M_{4p-1,4q-1} \subset S^{4n-1}\]
are invariant under the actions of of $\mathrm{U}(1)$ in $S^{2n-1}$ and $\mathrm{Sp}(1)$ in $S^{4n-1}$, so they induce minimal hypersurfaces $M^*_{p,q} \subset \C P^{n-1}$ and $M_{p,q}^{**} \subset \mathbb{H} P^{n-1}$, respectively. Again (see \cite[Chapter I, \S 6]{LawsonHsiang}) these minimal hypersurfaces are separating and their Jacobi fields are Killing-Jacobi fields, so they also satisfy the conditions of Theorem \ref{main}. Other examples include totally geodesic hypersurfaces in compact rank one symmetric spaces \cite{Ohnita}. %and in the complex quadratic hypersurface $Q_m = \mathrm{SO}(m+2)/\mathrm{SO}(2) \times \mathrm{SO}(m)$ for $m>1$ \cite{Qiang}.

An important family of homogeneous $3$-manifolds are given by the \emph{Berger spheres}, which are described by a one-parameter family $\{g_\tau\}_{\tau>1}$ metrics on $S^3$, with $\tau =1$ corresponding to the round metric. There are minimal tori in $(S^3,g_\tau)$ which are analogues of the Clifford tori in the round $3$-sphere. These tori are part of a one-parameter family of constant mean curvature tori in $(S^3,g_\tau)$, and all their Jacobi fields are Killing-Jacobi fields, see \cite{Levi}.

In another direction, there are minimal hypersurfaces of every genus $\geq 2$ in $S^3$ for which our main result can be applied. They are the minimal hypersurfaces $\xi_{1,m}$ constructed by Lawson in \cite{Lawson}, which have genus $m$. As recently proved by Kapouleas and Wiygul \cite[Theorem 6.21]{KapouleasWiygul}, the Jacobi fields of $\xi_{1,m}$ are Killing-Jacobi fields, for all $m\geq 2$.

This last example gives us an interesting consequence about the topology of level sets of solutions of \eqref{AC} in $S^3$. In a general context, the level sets of solutions of the Allen-Cahn equation may be quite complicated, even if we impose that the solutions are bounded. For instance, in \cite{EncisoSalas}, Enciso and Peralta-Salas proved that $\R^n$, $n\geq 4$, contains bounded entire solutions whose nodal sets have the topology of any closed hypersurface in $\R^n$. Their construction relies on the flexibility of infinite-index solutions of \eqref{AC}.

Contrastingly, any family of solutions of \eqref{AC} in $S^3$ with uniform energy and index bounds have nodal sets which converge graphically to limit interface, as proved by Chodosh and Mantoulidis, see Theorem 4.1 and Lemma 5.4 in \cite{OC}. In light of these facts, an application of Theorem \ref{main} to the family of Lawson surfaces $\xi_{1,m}$ in $S^3$ yields

\begin{corollary}
For any $m \in \N$, there exists $\e_0=\e_0(m)>0$ with the property that for each $\e \in (0,\e_0)$ there is a solution $u_\e$ of the Allen-Cahn equation \eqref{AC} on $S^3$ such that $u_\e^{-1}(0)$ is a surface $\Sigma_\e$ of genus $m$ in $S^3$. As $\e \downarrow 0$, this surface converges to a genus $m$ Lawson surface in $S^3$, in a $C^{2,\alpha}$ graphical sense.
\end{corollary}

\subsection{Local Killing fields and a generalization}

As noted in \cite{PR} and as pointed out in Section \ref{intro}, in some situations, the techniques developed by Pacard and Ritor\'e can be adapted to produce solutions of \eqref{AC} even when $\Gamma$ has nontrivial Jacobi fields. An example is $M=S^{n+1}$, the $(n+1)$-dimensional with endowed with the standard round metric, when $\Gamma$ is an equator (the intersection of $M$ with a plane in $\R^{n+1}$ through the origin).

However, if we consider a slight perturbation $\tilde g$ of the round metric around in $S^{n+1}$ supported away from $\Gamma$ (e.g. in a small neighborhood of the north pole), this approach may fail, as pointed out in \cite[p. 370]{PR}. Note that any such perturbation does not change the Jacobi fields of $\Gamma$, which are now given by the normal component of the \emph{local} Killing fields $X_1,\ldots, X_J$, meaning they satisfy $\mathcal{L}_{X_i}\tilde g=0$ in a neighborhood of $\Gamma$. Since the only step of the proof of Theorem \ref{main} in which the degeneracy of $\Gamma$ may pose a problem is the solution of the linear problem associated to the Jacobi operator, it is reasonable to believe that we can still find solutions of \eqref{AC} whose nodal sets accumulate around $\Gamma$.

In some situations, this can be achieved by suitably extending these local Killing fields to vector fields $\tilde X_1,\ldots, \tilde X_j$ defined on $M$, in a way that the conclusion of Proposition \ref{integral} still holds. If $\tilde g$ is an \emph{analytic} Riemannian metric on $S^{n+1}$, we can use the classical work of Nomizu \cite{nomizu}, which allows us to extend $X_1,\ldots, X_J$ to globally defined Killing fields -- this depends on the fact that $S^{n+1}$ is simply connected. This argument shows that the hypotheses of Theorem \ref{main} are satisfied, so we obtain solutions $u_\e$ for all $\e \in (0,\e_0)$, where $\e_0$ may depend on $\tilde g$.

We may extend this argument to an arbitrary perturbation of the round metric as follows. A close inspection of Section \ref{exist} (see also the Proof of Theorem 7.3 in \cite{OC}) shows that the smallness of $\e_0$ depends on the constants $C_1$, $C_2$, $C_0$, $C_3$, $\tilde C_2$ and $\tilde C$ in a way that they remain controlled for $C^{1,\alpha}$ perturbations of the metric $g$ which are supported away from $\Gamma$. More precisely, we may replace these constants so that they depend on a positive constant $\eta>0$ instead of $g$, and the same estimates hold for any metric $\tilde g$ such that $\|\tilde g - g\|_{C^{1,\alpha}(M \setminus \mathcal V_1)}<\eta$, where $\mathcal{V}_1$ is a fixed (but arbitrary) neighborhood of $\Gamma$ containing $\mathcal{V}$, where the Fermi coordinates are well-defined. Since real analytic Riemannian metrics are dense in the space of smooth metrics, the existence of solutions for general perturbations follows from the argument above.

In general, if $M$ is a simply connected manifold and $g$ is a analytical metric, then the same reasoning above yields

\begin{theorem} \label{resposta}
Assume $(M^{n+1},g)$ and $\Gamma \subset M$ are as in Theorem \ref{main}. Assume further that $g$ is (real) analytic, and that $M$ is simply connected. For any Riemannian metric $\tilde g$ such that $\tilde g = g$ in a neighborhood of $\Gamma$, there exists $\e_0>0$ such that for any $\e \in (0,\e_0)$ there is a solution $u_\e$ of \eqref{AC} such that $u_\e$ converges uniformly to $1$ (respectively to $-1$) on compact subsets of $M_+$ (respectively $M_{-}$), and
	\[E_\e(u_\e) \to \frac{1}{2\sigma}\Area(\Gamma) \quad \mbox{as} \quad \e \to 0.\]
Moreover, the Morse index $m(u_\e)$ and the nullity $n(u_\e)$ of $u_\e$ satisfy
	\[m(u_\e) = \ind(J_\Gamma) \qquad \mbox{and} \qquad n(u_\e) = \Null(J_\Gamma). \] 
\end{theorem}

\subsection{Final remarks} The work of del Pino, Kowalczyk, Wei and Yang \cite{PKWY} established the existence of solutions of \eqref{AC} in a compact manifold whose level sets accumulate on a separating nondegenerate minimal hypersurface $\Gamma$ with multiple transition layers. These solutions are notably different from the solutions constructed by Pacard and Ritor\'e \cite{PR}, or the solutions obtained in this paper. Due to resonance phenomena, the existence of these multiple-layer solutions is ensured only for a sequence of positive $\e$ which converges to $0$. Furthermore, their Morse indices blow up as $\e$ decreases, see \cite[Example 5.2]{OC} (in contrast, Agudelo, del Pino and Wei \cite{Agudelo1} produced entire index 1 solutions in $\R^3$ with multiple catenoidal layers as nodal sets).

In this context, one needs also to assume that the potential of the Jacobi field, that is $|A_{\Gamma}|^2 + \Ric(\nu,\nu)$, is positive along $\Gamma$. This condition is related to the interaction between the layers, which is described in terms of a Jacobi-Toda system. In view of these observations and Theorem \ref{main}, an interesting question is whether one can produce multiple layers solutions replacing the nondegeneracy of the Jacobi operator by the hypothesis of Theorem \ref{main}, namely that $\ker(J_\Gamma)$ is generated by Killing-Jacobi fields. We point out that a similar issue was tackled by Agudelo, del Pino and Wei in \cite{Agudelo1}, where they produced entire, index 1, axially symmetric solutions in $\R^3$ with multiple catenoidal layers as nodal sets.

\bibliography{main}

\begin{thebibliography}{10}

\bibitem{Agudelo1}
{\sc Agudelo, O., del Pino, M., and Wei, J.}
\newblock Solutions with multiple catenoidal ends to the {A}llen-{C}ahn
  equation in {$\mathbb{R}^3$}.
\newblock {\em J. Math. Pures Appl. (9) 103}, 1 (2015), 142--218.

\bibitem{Agudelo2}
{\sc Agudelo, O., Del~Pino, M., and Wei, J.}
\newblock Catenoidal layers for the {A}llen-{C}ahn equation in bounded domains.
\newblock {\em Chin. Ann. Math. Ser. B 38}, 1 (2017), 13--44.

\bibitem{AllenCahn}
{\sc Allen, S.~M., and Cahn, J.~W.}
\newblock A microscopic theory for antiphase boundary motion and its
  application to antiphase domain coarsening.
\newblock {\em Acta Metallurgica 27}, 6 (1979), 1085 -- 1095.

\bibitem{Almgren}
{\sc Almgren, Jr., F.~J.}
\newblock The homotopy groups of the integral cycle groups.
\newblock {\em Topology 1\/} (1962), 257--299.

\bibitem{BPS}
{\sc Bettiol, R.~G., Piccione, P., and Santoro, B.}
\newblock Deformations of free boundary cmc hypersurfaces.
\newblock {\em J. Geom. Anal. 27}, 4 (2017), 3254--3284.

\bibitem{ChanWei}
{\sc Chan, H., and Wei, J.}
\newblock On {D}e {G}iorgi's conjecture: Recent progress and open problems.
\newblock {\em Sci. China Math. 61}, 11 (2018), 1925--1946.

\bibitem{OC}
{\sc Chodosh, O., and Mantoulidis, C.}
\newblock Minimal surfaces and the {A}llen-{C}ahn equation on 3-manifolds:
  index, multiplicity, and curvature estimates.
\newblock {\em arXiv:1803.02716 [math.DG]\/} (2018).

\bibitem{DeGiorgi}
{\sc De~Giorgi, E.}
\newblock Convergence problems for functionals and operators.
\newblock In {\em Proceedings of the {I}nternational {M}eeting on {R}ecent
  {M}ethods in {N}onlinear {A}nalysis ({R}ome, 1978)\/} (1979), Pitagora,
  Bologna, pp.~131--188.

\bibitem{Levi}
{\sc De~Lima, L.~L., De~Lira, J.~H., and Piccione, P.}
\newblock Bifurcation of {C}lifford tori in {B}erger 3-spheres.
\newblock {\em Q. J. Math. 65}, 4 (2014), 1345--1362.

\bibitem{PKW-DG}
{\sc del Pino, M., Kowalczyk, M., and Wei, J.}
\newblock On {D}e {G}iorgi's conjecture in dimension {$N\geq 9$}.
\newblock {\em Ann. of Math. (2) 174}, 3 (2011), 1485--1569.

\bibitem{PKW}
{\sc del Pino, M., Kowalczyk, M., and Wei, J.}
\newblock Entire solutions of the {A}llen-{C}ahn equation and complete embedded
  minimal surfaces of finite total curvature in {$\mathbb R^3$}.
\newblock {\em J. Differ. Geom. 93}, 1 (2013), 67--131.

\bibitem{PKWY}
{\sc del Pino, M., Kowalczyk, M., Wei, J., and Yang, J.}
\newblock Interface foliation near minimal submanifolds in {R}iemannian
  manifolds with positive {R}icci curvature.
\newblock {\em Geom. Funct. Anal. 20}, 4 (2010), 918--957.

\bibitem{EncisoSalas}
{\sc Enciso, A., and Peralta-Salas, D.}
\newblock Bounded solutions to the {A}llen-{C}ahn equation with level sets of
  any compact topology.
\newblock {\em Anal. PDE 9}, 6 (2016), 1433--1446.

\bibitem{GasparInner}
{\sc Gaspar, P.}
\newblock The second inner variation of energy and the {M}orse index of limit
  interfaces.
\newblock {\em J. Geom. Anal.\/} (Jan 2019).

\bibitem{GasparGuaraco}
{\sc Gaspar, P., and Guaraco, M.~A.}
\newblock The {A}llen-{C}ahn equation on closed manifolds.
\newblock {\em Calc. Var. Partial Differ. Equ. 57}, 4 (2018), 101.

\bibitem{GGWeyl}
{\sc Gaspar, P., and Guaraco, M.~A.}
\newblock The {W}eyl {L}aw for the phase transition spectrum and density of
  limit interfaces.
\newblock {\em Geometric and Functional Analysis 29}, 2 (Apr 2019), 382--410.

\bibitem{Guaraco}
{\sc Guaraco, M.~A.}
\newblock Min--max for phase transitions and the existence of embedded minimal
  hypersurfaces.
\newblock {\em J. Differ. Geom. 108}, 1 (2018), 91--133.

\bibitem{Gurtin}
{\sc Gurtin, M.~E.}
\newblock Some results and conjectures in the gradient theory of phase
  transitions.
\newblock In {\em Metastability and incompletely posed problems ({M}inneapolis,
  {M}inn., 1985)}, vol.~3 of {\em IMA Vol. Math. Appl.} Springer, New York,
  1987, pp.~135--146.

\bibitem{Hiesmayr}
{\sc Hiesmayr, F.}
\newblock Spectrum and index of two-sided {A}llen-{C}ahn minimal hypersurfaces.
\newblock {\em Commun. Part. Differ. Equ. 43}, 11 (2018), 1541--1565.

\bibitem{Eduardo}
{\sc Hitomi, E.}
\newblock On the concentration for a singularly perturbed problem with
  nonlinear {N}eumann boundary condition.
\newblock {\em arXiv:1903.03185 [math.DG]\/} (2019).

\bibitem{LawsonHsiang}
{\sc Hsiang, W.-y., and Lawson, Jr., H.~B.}
\newblock Minimal submanifolds of low cohomogeneity.
\newblock {\em J. Differ. Geom. 5\/} (1971), 1--38.

\bibitem{HutchinsonTonegawa}
{\sc Hutchinson, J.~E., and Tonegawa, Y.}
\newblock Convergence of phase interfaces in the van der
  {W}aals-{C}ahn-{H}illiard theory.
\newblock {\em Calc. Var. Partial Differ. Equ. 10}, 1 (2000), 49--84.

\bibitem{IMN}
{\sc Irie, K., Marques, F., and Neves, A.}
\newblock Density of minimal hypersurfaces for generic metrics.
\newblock {\em Ann. of Math. (2) 187}, 3 (2018), 963--972.

\bibitem{KapouleasWiygul}
{\sc Kapouleas, N., and Wiygul, D.}
\newblock The index and nullity of some {L}awson surfaces.
\newblock {\em arXiv:1904.05812 [math.DG]\/} (2019).

\bibitem{Lawson}
{\sc Lawson, Jr., H.~B.}
\newblock Complete minimal surfaces in {$S^{3}$}.
\newblock {\em Ann. of Math. (2) 92\/} (1970), 335--374.

\bibitem{LMN}
{\sc Liokumovich, Y., Marques, F.~C., and Neves, A.}
\newblock Weyl law for the volume spectrum.
\newblock {\em Ann. of Math. (2) 187}, 3 (2018), 933--961.

\bibitem{Willmore}
{\sc Marques, F.~C., and Neves, A.}
\newblock Min-max theory and the {W}illmore conjecture.
\newblock {\em Ann. of Math. (2) 179}, 2 (2014), 683--782.

\bibitem{MNIndex}
{\sc Marques, F.~C., and Neves, A.}
\newblock Morse index and multiplicity of min-max minimal hypersurfaces.
\newblock {\em Cambridge J. Math. 4}, 4 (2016), 463--511.

\bibitem{MNMult}
{\sc Marques, F.~C., and Neves, A.}
\newblock Morse index of multiplicity one min-max minimal hypersurfaces.
\newblock {\em arXiv:1803.04273 [math.DG]\/} (2018).

\bibitem{MNS}
{\sc Marques, F.~C., Neves, A., and Song, A.}
\newblock Equidistribution of minimal hypersurfaces for generic metrics.
\newblock {\em Invent. Math. 216}, 2 (2019), 421--443.

\bibitem{Modica}
{\sc Modica, L.}
\newblock The gradient theory of phase transitions and the minimal interface
  criterion.
\newblock {\em Arch. Ration. Mech. Anal 98}, 2 (1987), 123--142.

\bibitem{nomizu}
{\sc Nomizu, K.}
\newblock On local and global existence of {K}illing vector fields.
\newblock {\em Ann. of Math. (2) 72\/} (1960), 105--120.

\bibitem{Ohnita}
{\sc Ohnita, Y.}
\newblock On stability of minimal submanifolds in compact symmetric spaces.
\newblock {\em Compositio Math. 64}, 2 (1987), 157--189.

\bibitem{PA}
{\sc Pacard, F.}
\newblock The role of minimal surfaces in the study of the {A}llen-{C}ahn
  equation.
\newblock In {\em Geometric analysis: partial differential equations and
  surfaces}, vol.~570 of {\em Contemp. Math.} Amer. Math. Soc., Providence, RI,
  2012, pp.~137--163.

\bibitem{PR}
{\sc Pacard, F., and Ritor{\'e}, M.}
\newblock From constant mean curvature hypersurfaces to the gradient theory of
  phase transitions.
\newblock {\em J. Differ. Geom. 64}, 3 (2003), 359--423.

\bibitem{Passaseo}
{\sc Passaseo, D.}
\newblock Multiplicity of critical points for some functionals related to the
  minimal surfaces problem.
\newblock {\em Calc. Var. Partial Differ. Equ. 6}, 2 (1998), 105--121.

\bibitem{Pitts}
{\sc Pitts, J.~T.}
\newblock {\em Existence and regularity of minimal surfaces on {R}iemannian
  manifolds}, vol.~27 of {\em Mathematical Notes}.
\newblock Princeton University Press, Princeton, N.J.; University of Tokyo
  Press, Tokyo, 1981.

\bibitem{Song}
{\sc Song, A.}
\newblock Existence of infinitely many minimal hypersurfaces in closed
  manifolds.
\newblock {\em arXiv:1806.08816 [math.DG]\/} (2018).

\bibitem{Sternberg}
{\sc Sternberg, P.}
\newblock The effect of a singular perturbation on nonconvex variational
  problems.
\newblock {\em Arch. Ration. Mech. Anal 101}, 3 (1988), 209--260.

\bibitem{WangWei}
{\sc Wang, K., and Wei, J.}
\newblock Finite morse index implies finite ends.
\newblock {\em Commun. Pure Appl. Math. 72}, 5 (2019), 1044--1119.

\bibitem{White1}
{\sc White, B.}
\newblock The space of minimal submanifolds for varying {R}iemannian metrics.
\newblock {\em Indiana Univ. Math. J. 40}, 1 (1991), 161--200.

\bibitem{White2}
{\sc White, B.}
\newblock On the bumpy metrics theorem for minimal submanifolds.
\newblock {\em Amer. J. Math. 139}, 4 (2017), 1149--1155.

\bibitem{Zhou}
{\sc Zhou, X.}
\newblock On the multiplicity one conjecture in min-max theory.
\newblock {\em arXiv:1901.01173 [math.DG]\/} (2019).

\end{thebibliography}
\bibliographystyle{acm}

\end{document}